\documentclass[a4paper,11pt,reqno]{amsart}%
\usepackage{amsfonts}
\usepackage{amsthm}
\usepackage{hyperref}
\usepackage{amssymb}
\usepackage{graphicx, texdraw,color}
\usepackage{float}
\usepackage{todonotes}
\usepackage{eqparbox,array}
\usepackage{dcolumn}
\usepackage{amsmath}
\usepackage{verbatim}%
\usepackage{tikz}
\usetikzlibrary{arrows}

\newcolumntype{r}{D{.}{.}{-1}}

\newtheorem{theorem}{Theorem}[section]

\theoremstyle{plain}

\newtheorem{corollary}[theorem]{Corollary}

\theoremstyle{definition}
\newtheorem{definition}[theorem]{Definition}
\newtheorem{example}[theorem]{Example}
\newtheorem{construction}[theorem]{Construction}

\theoremstyle{plain}
\newtheorem{lemma}[theorem]{Lemma}

\newtheorem{proposition}[theorem]{Proposition}

\theoremstyle{remark}
\newtheorem{remark}[theorem]{Remark}
\numberwithin{equation}{section}

\DeclareMathOperator{\trop}{trop}
\DeclareMathOperator{\Aut}{Aut}
\DeclareMathOperator{\ft}{ft}
\DeclareMathOperator{\mult}{mult}
\DeclareMathOperator{\id}{id}

\newcommand{\C}{\mathbb C}
\def\H{\mathbb H}
\def\SL{{\rm SL}}

\newcommand{\cC}{\mathcal{C}}
\newcommand{\cE}{\mathcal{E}}
\newcommand {\PP}{{\mathbb P}}
\newcommand {\RR}{{\mathbb R}}

\newcommand {\Perm}{{\mathbb S}}
\newcommand {\NN}{{\mathbb N}}
\newcommand{\Z}{\mathbb Z}

\begin{document}
\title[Tropical mirror symmetry for elliptic curves]{Tropical mirror symmetry for elliptic curves}
\author[Janko B\"{o}hm, Kathrin Bringmann, Arne Buchholz, Hannah Markwig]{Janko B\"{o}hm, Kathrin Bringmann, Arne Buchholz, Hannah Markwig}
\address{Janko B\"ohm, Fachbereich Mathematik,
Universit\"{a}t Kaiserslautern, Postfach 3049, 67653 Kaiserslautern, Germany}
\email{boehm@mathematik.uni-kl.de}
\address{Kathrin Bringmann, Mathematical Institute, University of Cologne, Weyertal 86-90, 50931 Cologne, Germany}
\email{kbringma@math.uni-koeln.de}
\address {Arne Buchholz, Cluster of Excellence $M^2CI$, Fachrichtung Mathematik,
  Universit\"at des Saarlandes, Postfach 151150, 66041 Saarbr\"ucken, Germany}
\email {buchholz@math.uni-sb.de}
\address {Hannah Markwig, Universit\"at des Saarlandes\\ Fachrichtung Mathematik\\ Postfach 151150, 66041 Saarbr\"ucken\\ Germany}
\email {hannah@math.uni-sb.de}


\thanks{2010 Math subject classification: Primary 14J33, 14N35, 14T05, 11F11, 81T18; Secondary 14H30, 14N10, 14H52, 14H81}

\begin{abstract}
Mirror symmetry relates Gromov-Witten invariants of an elliptic curve with certain integrals over Feynman graphs \cite{Dij95}. We prove a tropical generalization of mirror symmetry for elliptic curves, i.e., a statement relating certain labeled Gromov-Witten invariants of a tropical elliptic curve to more refined Feynman integrals.
This result easily implies the tropical analogue of the mirror symmetry statement mentioned above and, using the necessary Correspondence Theorem, also the mirror symmetry statement itself. 
In this way, our tropical generalization leads to an alternative proof of mirror symmetry for elliptic curves. We believe that our approach via tropical mirror symmetry naturally carries the potential of being generalized to more adventurous situations of mirror symmetry. 
Moreover, our tropical approach has the advantage that all involved invariants are easy to compute.
Also, as a side product, we can give a combinatorial characterization of Feynman graphs for which the corresponding integrals are zero. More generally, the tropical mirror symmetry theorem gives a natural interpretation of the A-model side (i.e., the generating function of Gromov-Witten invariants) in terms of a sum over Feynman graphs. Our technique for computing Feynman integrals inspired Goujard and M\"oller to prove further quasimodularity results \cite{GM16}. By our theorem, this quasimodularity result becomes meaningful on the A-model side as well. 
Our theoretical results are complemented by a \textsc{Singular} package including several procedures that can be used to compute Hurwitz numbers of the elliptic curve as integrals over Feynman graphs.

\end{abstract}
\maketitle
\tableofcontents

\section{Introduction}

 Mirror symmetry is a deep symmetry relation motivated by dualities in string theory.
 Many results and conjectures of different flavors are related to mirror symmetry (see e.g.\ \cite{ FTY13, FLTZ12,GS06, HMS09, MSTG10,RIMS10,KM13}).
 Here, we focus on statements relating Gromov-Witten invariants of a variety $X$ to certain integrals on a mirror partner $X^\vee$.

Tropical geometry has proved to be an interesting new tool for mirror symmetry (see e.g.\ \cite{Abo09, Gro09, GS06, GS07, MSTG10}), in particular through the 
famous Gross-Siebert programme offering a new tool to prove such 
relations while at the same time aiming at the construction of new 
mirror pairs.
 Our paper can be viewed as a sequel and extension of Gross' paper \cite{Gro09}, where he provides tropical methods for the study of mirror symmetry of $\PP^2$.
 The main purpose of his paper is of a philosophical nature: he suggests tropical geometry as a new and worthwhile method for the study of mirror symmetry.
  More precisely, you can find (a version of) the following triangle in the introduction of \cite{Gro09}:
 

\[\scalebox{0.90}{
\begin{tikzpicture}[<->,>=stealth',shorten >=1pt,auto]
\coordinate (a) at (0,0);
\coordinate (b) at (-3.4,3.5);
\coordinate (c) at (3.4,3.5);
\node(1) at (a)  {\begin{tabular}{c}tropical\\ GW-invariants\end{tabular}};
\node(2) at (b) {\begin{tabular}{c}Gromov-Witten\\ invariants\end{tabular}};
\node(3) at (c)  {\begin{tabular}{c}Feynman\\ integrals\end{tabular}};
\path[every node/.style={}]
(1) edge node[left,sloped, anchor=center] {\begin{tabular}{c}Correspondence\\ Theorem\end{tabular}} (2)
(2) edge [dashed] node {Mirror symmetry} (3)
(3) edge node[right,sloped, anchor=center] {} (1);
\end{tikzpicture}
}
\]

The relation between Gromov-Witten invariants and integrals (the top arrow) is a consequence of mirror symmetry. 
Tropical geometry comes in naturally, because there are many instances of Correspondence Theorems that relate Gromov-Witten invariants with their tropical analogues (the first of these is due to Mikhalkin \cite{Mi03}).
The connection between tropical geometry and integrals is in general yet to be understood.

Gross studies the triangle in the situation of $\PP^2$: here, the mirror is $(\mathbb{C}^\ast)^2$.  For this case, statements relating Gromov-Witten invariants to integrals are known already \cite{Bar00}; but again, the purpose is to a lesser extent to give a new proof of this mirror symmetry relation, but to outline a new path for future progress in mirror symmetry.
In the case of $\PP^2$, the mirror symmetry relation involves descendant Gromov-Witten invariants of $\PP^2$. Only a partial Correspondence Theorem is proved to relate some of these invariants to their tropical counterparts \cite{MR08}. Gross concentrates on proving the right arrow in his situation, i.e., he provides a natural connection between integrals and tropical Gromov-Witten invariants. This connection very roughly relates monomials in a big generating function that yield a nonzero contribution to an integral with pieces of tropical curves that glue to one big tropical curve satisfying the requirements. The heart of the argument is thus a purely combinatorial hunt of monomials respectively pieces of tropical curves, and the fact that both sides can be boiled down to combinatorics that fits together very naturally strongly recommends the tropical approach for future experiments in mirror symmetry.
However, since the existing Correspondence Theorems are not sufficient yet to cover the whole situation needed for mirror symmetry of $\PP^2$, this exciting new approach does not give an alternative proof of the mirror symmetry statement for $\PP^2$ (the top arrow) yet.

\vspace{0.2cm}

It is our purpose to demonstrate that tropical geometry can actually lead to a complete alternative proof of mirror symmetry, and that this approach is very natural and requires not much more than a careful analysis of the underlying combinatorics. We prove a tropical mirror symmetry theorem for tropical elliptic curves that in particular implies mirror symmetry for elliptic curves. Furthermore, we pursue our approach to prove new results about the involved generating functions, most importantly the quasimodularity of certain Feynman integrals resp.\ generating functions of Hurwitz numbers.

\vspace{0.2cm}

We study the triangle above for elliptic curves.
As in the case of $\PP^2$, mirror symmetry of elliptic curves is known and can to the best of our knowledge even be considered as folklore in mirror symmetry \cite{Dij95}, i.e., in principle the top arrow is taken care of already (see Theorem \ref{thm-mirror}). The known proof is inspired by physics and uses quantum field theory.

We provide the alternative route via the left and right arrows.
The Gromov-Witten invariants involved in the upper left vertex of the triangle are nothing but Hurwitz numbers --- numbers of covers of an elliptic curve having simple ramification above some fixed branch points.
The integrals in the upper right vertex are certain integrals over Feynman graphs.

Correspondences between Hurwitz numbers and their tropical counterparts have been studied \cite{BBM10, CJM10}, and an easy generalization yields a Correspondence Theorem that is suitable for our situation (see Theorem \ref{thm-corres}). As in the case of $\PP^2$, more interesting is the right arrow.
It turns out that a more general formulation of mirror symmetry is more natural on the tropical side.
We prove a tropical mirror symmetry theorem (Theorem \ref{thm-refined}) relating numbers of labeled tropical covers with refined Feynman integrals. 
A careful analysis of the combinatorial principles underlying the count of labeled tropical covers on the one hand and nonzero contributions to refined integrals over Feynman graphs on the other hand reveals that they can be related naturally. The right arrow then follows from our tropical mirror symmetry theorem, see Theorem \ref{thm-tropmirror}. 

To sum up, for the case of elliptic curves, we complete the picture of the triangle as follows:

\[\scalebox{0.90}{
\begin{tikzpicture}[<->,>=stealth',shorten >=1pt,auto]
\coordinate (a) at (0,0);
\coordinate (b) at (-3.4,5);
\coordinate (c) at (3.4,5);
\coordinate (d) at (9.4,5);
\coordinate (e) at (6,0);
\coordinate (f) at (5.25,2.5);
\coordinate (g) at (3.5,2.5);
\node(1) at (a)  {\begin{tabular}{c}tropical\\ Hurwitz invariants\end{tabular}};
\node(2) at (b) {\begin{tabular}{c}Hurwitz\\ numbers\end{tabular}};
\node(3) at (c)  {\begin{tabular}{c}Feynman\\ integrals\end{tabular}};
\node(4) at (d)  {\begin{tabular}{c}refined \\ Feynman integrals\end{tabular}};
\node(5) at (e)  {\begin{tabular}{c}numbers of \\ labeled tropical covers\end{tabular}};
\node(6) at (f)  {};
\node(7) at (g)  {};
\path[every node/.style={}]
(1) edge node[left,sloped, anchor=center] {\begin{tabular}{c}Correspondence\\ Theorem  \ref*{thm-corres} \end{tabular}} (2)
(2) edge [dashed] node {Mirror symmetry} (3)
(3) edge node[right,sloped, anchor=center] {\begin{tabular}{c}Theorem \ref*{thm-tropmirror}\\ \phantom{m}\end{tabular}} (1)
(4) edge node[right,sloped, anchor=center] {\begin{tabular}{c} Tropical Mirror \\ Symmetry Thm. \ref*{thm-refined}\end{tabular}} (5)
(6) edge[double distance=4pt,->,>=implies] node[right,sloped, anchor=center] {} (7);
\end{tikzpicture}
}
\]

We provide proofs for all solid arrows, in particular this implies the dashed arrow, the mirror symmetry statement for elliptic curves. We thus provide a complete and natural combinatorial proof of mirror symmetry for elliptic curves by means of tropical geometry. 


Note that our case is the first instance where tropical mirror symmetry is understood for a compact Calabi-Yau variety, and for arbitrary genus Gromov-Witten invariants.

\vspace{0.2cm}

Furthermore, our study of Feynman integrals can be pursued in another direction related to number theory.
In number theory, the study of modular forms and quasimodular forms is of great interest. Let us just mention a few examples. In VOA theory, $1$-point functions are quasimodular \cite{DMN}. 
Conversely, every quasimodular form can be written in terms of $1$-point functions. Further examples occur in the study of $s\ell(n|n)^{\land}$ characters \cite{BFM} or in curve counting problems on Abelian surfaces (see e.g. \cite{AR}).
Knowing modularity results presents one with deep tools, like Sturm's Theorem or the valence formula, to prove identities, congruences, or the asymptotic behavior of arithmetically interesting functions. 
Most relevant for this paper is that the generating function for Hurwitz numbers of elliptic curves is a quasimodular form (see e.g. \cite{Dij95}, \cite{KZ95}). 
Via mirror symmetry, this generating function equals a sum of Feynman integrals, where each summand corresponds to a graph. To the best of our knowledge, the quasimodularity of the individual summands has not been known. It follows from \cite{GM16} that the summands $I_{\Gamma, \Omega}$ as defined in Definition \ref{def-int}, Equation (\ref{eq-Igamma}) are quasimodular (of mixed weight).
Other Feynman integrals related to mirror symmetry have been proved to be quasimodular forms, too \cite{Li12}. 

Note that the Tropical Mirror Symmetry Theorem naturally gives an interpretation of the generating function of Hurwitz numbers in terms of a sum over Feynman graphs --- we can sum over all tropical covers whose combinatorial type is a fixed Feynman graph.
In this way, the quasimodularity result becomes meaningful on the A-model side as well: The generating function of Hurwitz numbers for a fixed combinatorial type is quasimodular (of mixed weight). 

As a side product we also give a combinatorial characterization of graphs whose corresponding Feynman integral is zero: we prove in Corollary \ref{cor-bridgegraph} that a graph yields a zero Feynman integral if and only if it contains a bridge. 


Our theoretical results are complemented by the package {\sc ellipticcovers.lib} \cite{BBBM} for the computer algebra system {\sc Singular} \cite{DGPS}, which contains several procedures that can be used to compute Hurwitz numbers of the elliptic curve as integrals over Feynman graphs. 


\vspace{0.4cm}

Our paper is organized as follows.
In Section \ref{subsec-mirror} we define the invariants and state the Mirror Symmetry Theorem \ref{thm-mirror}, i.e., the top arrow in the triangle above.
In Section \ref{subsec-tropcovers} we define tropical Hurwitz numbers and state the Correspondence Theorem \ref{thm-corres}, i.e., the left arrow in the triangle. We also state Theorem \ref{thm-tropmirror}, the right arrow in the triangle. 
In Section \ref{subsec-labeled}, we introduce the invariants involved in our generalization of theorem \ref{thm-tropmirror}: labeled tropical covers (tropical covers with some extra structure) and a more refined version of the Feynman integrals. 
In the same section, we can finally also state our main result, the Tropical Mirror Symmetry Theorem \ref{thm-refined} involving labeled tropical covers and refined Feynman integrals. We also deduce Theorem \ref{thm-tropmirror} using Theorem \ref{thm-refined}.
In Section \ref{sec-prop}, we take a closer look at the involved Feynman integrals.
A central statement is Theorem \ref{thm-prop} that expresses the propagator as a series in a special nice way. This helps us to boil down the computation of the refined Feynman integrals (using a coordinate change and the residue theorem) to a purely combinatorial hunt of monomials in a big generating function. 
The precise statement of the monomial hunt can be found in Lemma \ref{lem-seriesconstterm}.
In Section \ref{subsec-bij} finally, we establish our bijection relating labeled tropical covers and nonzero contributions to refined Feynman integrals. The precise formulation can be found in Theorem \ref{thm-bij}, Theorem \ref{thm-refined} follows easily from this. Construction \ref{const-bij} gives a hands-on algorithm how to produce a tropical cover given a tuple contributing to the integral, Lemmas \ref{lem-themap} and \ref{lem-inverse} prove that this algorithm indeed produces a bijective map.
In Section \ref{sec-quasimod} we review the quasimodularity of the summands of the generating function of Hurwitz numbers corresponding to Feynman graphs \cite{GM16}.

In the Appendix, we prove the Correspondence Theorem
\ref{thm-corres}. It could be deduced from the Correspondence Theorem in
\cite{BBM10}, for the sake of completeness however we include our own proof here
which uses combinatorial methods involving the symmetric group.


\vspace{0.1in} \noindent\emph{Acknowledgements}. 
We would like to thank Arend Bayer, Erwan Brugall\'{e}, Elise Goujard, Albrecht Klemm, Laura Matusevich, Martin M\"oller, and Rainer Schulze-Pillot for helpful discussions. Part of this work was accomplished during the third and fourth author's stay at the Max-Planck-Institute for Mathematics in Bonn. We thank the MPI for hospitality.
The research of the second author is supported by the Alfried Krupp Prize for Young University Teachers of the Krupp foundation and the research leading to these results receives funding from the European Research Council under the European Union's Seventh Framework Programme (FP/2007-2013) / ERC Grant agreement n. 335220 - AQSER. 
The fourth author is partially supported by DFG-grant MA 4797/3-2. The first and fourth author acknowledge support of DFG-SPP 1489.
We thank an anonymous referee for useful comments on an earlier version of this work.

\section{Tropical Mirror Symmetry for elliptic curves}

\subsection{Mirror Symmetry for elliptic curves}\label{subsec-mirror}
In this subsection, we define the relevant invariants (i.e., Hurwitz numbers and Feynman integrals) and present a precise statement of the top arrow of the triangle in the introduction. 

Hurwitz numbers count branched covers of non-singular curves with a given
ramification
profile over fixed points. In this paper, we consider covers of elliptic curves. Hurwitz numbers are naturally topological invariants,
in particular they do not depend on the position of the branch points as long as
these
are pairwise different. 
Moreover, since all complex elliptic curves are
homeomorphic to the real torus, numbers of covers of an
elliptic curve do not depend on the choice of the base curve. 
We thus fix an arbitrary complex elliptic curve $\cE$.

Let $\cC$ be a non-singular curve of genus $g$ and $\phi:\cC\rightarrow\cE$ a
cover.
We denote by $d$ the degree of $\phi$, i.e., the
number of preimages of a generic point in $\cE$.
For our purpose, it is sufficient to
consider covers which are \emph{simply ramified}, that is, over any branch
point exactly two sheets of the map come together and all others stay separate.
In other words, the \emph{ramification profile} (i.e. the partition of the
degree indicating the multiplicities of the inverse images of a branch point) of
a simple branch point is $(2,1,\ldots,1)$. 
It follows from the Riemann-Hurwitz formula (see e.g.\ \cite{Har77}, Corollary IV.2.4) that a simply ramified cover of $\cE$ has exactly $2g-2$ branch points.
 Two covers $\varphi:\cC\rightarrow\cE$ and
$\varphi':\cC'\rightarrow\cE$ are isomorphic if there exists an isomorphism of
curves $\phi:\cC\rightarrow\cC'$ such that $\varphi'\circ\phi=\varphi$.

\begin{definition}[Hurwitz numbers]\label{def-alghurwitz}
Fix $2g-2$ points $p_1,\ldots,p_{2g-2}$ in $\cE$. We define the \emph{Hurwitz number} $N_{d,g}$ to be the weighted number of (isomorphism classes of) simply ramified covers $\phi:\cC\rightarrow \cE$ of degree $d$, where $\cC$ is a connected curve of genus $g$, and the branch points of $\phi$ are the points $p_i$, $i=1,\ldots,2g-2$. We count each such cover $\phi$ with weight $|\Aut(\phi)|$.
\end{definition}
For more details on branched covers of elliptic curves and Hurwitz numbers, see e.g.\ \cite{RY10}.

\begin{remark}\label{rem-branchpointsmarked}
 Note that by convention, we mark the branch points $p_i$ in this definition. In the literature, you can also find definitions without this convention, leading to a factor of $(2g-2)!$ when compared with our definition.
\end{remark}

\begin{definition}
 We package the Hurwitz numbers of Definition \ref{def-alghurwitz} into a generating function as follows:
$$ F_g(q):=\sum_{d=1}^\infty N_{d,g} q^{2d}.$$
\end{definition}

The mirror symmetry statement relates the generating function $F_g(q)$ to certain integrals which we are going to define now. We start by defining the function which we are going to integrate.
\begin{definition}[The propagator]\label{def-prop}
 We define the \emph{propagator} $$P(z,q):=\frac{1}{4\pi^2}\wp(z,q)+\frac{1}{12}E_2(q^2)$$ in terms of the Weierstra\ss{}-P-function $\wp$ and the Eisenstein series $$E_2(q):=1-24\sum_{d=1}^\infty \sigma(d)q^d.$$ Here, $\sigma=\sigma_1$ denotes the sum-of-divisors function $\sigma(d)=\sigma_1(d)=\sum_{m|d}m$.
\end{definition}
The variable $q$ above should be considered as a coordinate of the moduli space of elliptic curves, the variable $z$ is the complex coordinate of a fixed elliptic curve. (More precisely, $q=e^{i\pi \tau}$, where $\tau \in \C$ is the parameter in the upper half plane in the well-known definition of the Weierstra\ss{}-P-function.)

\begin{definition}[Feynman graphs and integrals]\label{def-int}
 A \emph{Feynman graph} $\Gamma$ of genus $g$ is a trivalent connected graph of genus $g$. For a Feynman graph, we throughout fix a reference labeling $x_1,\ldots,x_{2g-2}$ of the $2g-2$ trivalent vertices and a reference labeling $q_1,\ldots,q_{3g-3}$ of the edges of $\Gamma$.

For an edge $q_k$ of $\Gamma$ connecting the vertices $x_i$ and $x_j$, we define a function
$$P_k:=P(z_i-z_j,q),$$ 
where $P$ denotes the propagator of Definition \ref{def-prop} (the choice of sign i.e., $z_i-z_j$ or $z_j-z_i$ plays no role, more about this in Section \ref{sec-prop}).
Pick a total ordering $\Omega$ of the vertices and starting points of the form $iy_1,\ldots, iy_{2g-2}$ in the complex plane, where the $y_j$ are pairwise different small real numbers.
We define integration paths $\gamma_1,\ldots,\gamma_{2g-2}$ by $$\gamma_j:[0,1]\rightarrow \mathbb{C}:t\mapsto iy_j+t,$$ such that the order of the real coordinates $y_j$ of the starting points of the paths equals $\Omega$. 
We then define the integral \begin{equation}I_{\Gamma,\Omega}(q):= \int_{z_j\in \gamma_j} \prod_{k=1}^{3g-3} \left(-P_k\right).\label{eq-Igamma}\end{equation}
Finally, we define
$$I_\Gamma(q)= \sum_{\Omega} I_{\Gamma,\Omega}(q),$$ where the sum runs over all $(2g-2)!$ orders of the vertices. 
\end{definition}

The following is the precise statement of the top arrow in the triangle in the introduction (see Theorem 3 of \cite{Dij95}):
\begin{theorem}[Mirror Symmetry for elliptic curves]\label{thm-mirror}
Let $g>1$. For the definition of the invariants, see \ref{def-alghurwitz} and \ref{def-int}. We have  $$ F_g(q)=\sum_{d=1}^\infty N_{d,g} q^{2d}= \sum_{\Gamma} I_\Gamma(q)\cdot \frac{1}{|\Aut(\Gamma)|},$$ where $\Aut(\Gamma)$ denotes the automorphism group of $\Gamma$ and the sum goes over all trivalent graphs $\Gamma$ of genus $g$.
\end{theorem}

\subsection{Tropical covers and Hurwitz numbers}\label{subsec-tropcovers}

We give a proof of Theorem \ref{thm-mirror} by a detour to tropical geometry and tropical mirror symmetry. We first relate Hurwitz numbers to tropical Hurwitz numbers.

\begin{definition}[Tropical curves]
 A (generic) \emph{tropical curve} $C$ (without ends) is a connected, finite, trivalent, metric graph.
Its
\emph{genus} is given by its first Betti number. The
\emph{combinatorial type} of a curve is its homeomorphism class, i.e., the
underlying graph without lengths on the edges.
\end{definition}

A tropical elliptic curve consists of one edge forming a circle of certain
length. We will fix a tropical elliptic curve $E$ (e.g. one having length $1$).
Moreover, to fix notation, in the following we denote by $C$ a tropical curve of genus $g$ and
combinatorial type $\Gamma$.

\begin{definition}[Tropical covers]\label{def-tropcover}
A map $\pi:C\rightarrow E$ is a \emph{tropical cover} of $E$ if it is
continuous,
non-constant, integer affine on each edge and respects the \emph{balancing
condition} at every vertex $P\in C$:

For an edge $e$ of $C$ denote by $w_e$ the \emph{weight of $e$}, i.e., the
absolute
value of the slope of $\pi_{|e}$. Consider a (small) open neighbourhood $U$ of
$p=\pi(P)$, then $U$ combinatorially consists of $p$ together with two rays $r_1$  and $r_2$, left and right of $p$.
Let $V$ be the connected component of $\pi^{-1}(U)$ which contains $P$. Then
$V$ consists of $P$ adjacent to three rays). Then $\pi$ is balanced at $P$ if
$$\sum_{R \mbox{{ \small maps to }} r_1}
w_R=\sum_{R \mbox{ {\small maps to }} r_2}w_R,$$ where $R$ goes over the three rays adjacent to
$P$ and rays inherit their weights from the corresponding edges.

The \emph{degree} of a cover is the weighted number of preimages of a generic
point: For all $p\in E$ not having any vertex of $C$ as preimage we have
$$d=\sum_{P\in C:\pi(P)=p} w_{e_P},$$ where $e_P$ is the edge of $C$ containing
$P$. The images of the vertices of $C$ are called the \emph{branch points} of
$\pi$.
\end{definition}

\begin{example}
 Figure \ref{fig-tropcover} shows a tropical cover of degree $4$ with a genus $2$
source curve. The red numbers close to the vertex $P$ are the weights of the corresponding edges, the black numbers denote the lengths. The cover is balanced at
$P$ since there is an edge of weight $3$ leaving in one direction and an
edge of
weight $2$ plus an edge of weight $1$ leaving in the opposite direction.
\end{example}

We can see that the length of an edge of $C$ is determined by its weight and
the length of its image. 
We will therefore in the following not specify edge lengths anymore.

\begin{figure}
 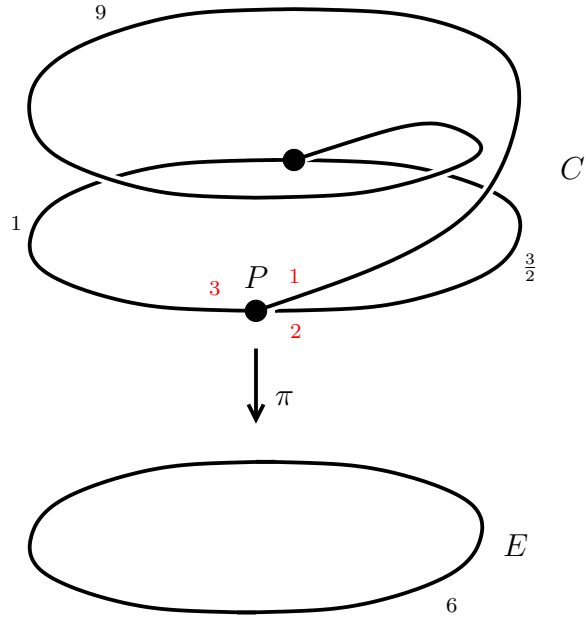
\caption[A tropical cover]{A tropical cover of degree $4$ with genus $2$ source
curve.}\label{fig-tropcover}
\end{figure}

\begin{definition}[Isomorphisms of curves and covers]
 An \emph{isomorphism of tropical curves} is an isometry of metric graphs. Two
covers $\pi:C\rightarrow E$ and $\pi':C'\rightarrow E$ are isomorphic if there
is an isomorphism of curves $\phi:C\rightarrow C'$ such that $\pi'\circ\phi=\pi$.
\end{definition}

As usual when counting tropical objects we have to weight them with a certain
multiplicity.

\begin{definition}[Multiplicities]\label{def-mult}
 The multiplicity of a cover $\pi:C\rightarrow E$ is defined to be
$$\mult(\pi):=\frac{1}{|\Aut(\pi)|}\prod_e w_e,$$ where the product goes over all edges $e$ of $C$ and $\Aut(\pi)$ is the automorphism group of $\pi$.
\end{definition}


\begin{definition}[Tropical Hurwitz numbers]\label{def-trophurwitz}
Fix branch points $p_1,\ldots,p_{2g-2}$ in the tropical, elliptic curve $E$.
The \emph{tropical Hurwitz number} $N_{d,g}^{\trop}$ is the weighted number of
isomorphism classes of degree $d$ covers $\pi:C\rightarrow E$ having their branch
points at the $p_i$, where $C$ is a curve of genus $g$ :
$$N_{d,g}^{\trop}:=\sum_{\pi:C\rightarrow E}\mult(\pi).$$
\end{definition}

The following is the precise formulation of the left arrow of the triangle in the introduction:
\begin{theorem}[Correspondence Theorem]\label{thm-corres}
The algebraic and tropical Hurwitz numbers of simply ramified covers of an elliptic curve coincide (see Definition \ref{def-alghurwitz} and \ref{def-trophurwitz}), i.e., we have 
 $$N_{d,g}^{\trop}=N_{d,g}.$$ 
\end{theorem}

This theorem could be proved by a mild generalization of the Correspondence Theorem for Hurwitz numbers in \cite{BBM10}. For the sake of completeness, we give our own proof, using more combinatorial methods involving the symmetric group. We present it in the Appendix \ref{ap-corres}.

Using the Correspondence Theorem \ref{thm-corres}, it is obviously sufficient to prove the following theorem in order to obtain a tropical proof of Theorem \ref{thm-mirror}. This theorem can be viewed as the right arrow in the triangle in the introduction.

\begin{theorem}[Tropical curves and integrals]\label{thm-tropmirror}
 Let $g>1$. For the definition of the invariants, see Definitions \ref{def-trophurwitz} and \ref{def-int}. We have  $$ \sum_{d=1}^\infty N_{d,g}^{\trop} q^{2d}= \sum_{\Gamma} I_\Gamma(q)\cdot \frac{1}{|\Aut(\Gamma)|},$$ where the sum goes over all trivalent graphs $\Gamma$ of genus $g$.
\end{theorem}

\subsection{Labeled tropical covers}\label{subsec-labeled}

To deduce Theorem \ref{thm-tropmirror}, we prove a more general statement that implies this theorem, namely our Tropical Mirror Symmetry Theorem for elliptic curves \ref{thm-refined}.
To state the result, we need to introduce labeled tropical covers, and a refined version of the Feynman integrals from above.
This more general tropical mirror symmetry theorem is more natural on the tropical side, since the combinatorics involved in the hunt of monomials contributing to the refined Feynman integrals resp.\ in counting labeled tropical covers can be related very naturally.

\begin{definition}[Labeled tropical covers]\label{def-labeledtropcover}
Let $\pi:C\rightarrow E$ be a tropical cover as in Definition \ref{def-tropcover}. 
Let $\Gamma$ be the combinatorial type of the tropical curve $C$.
We fix a reference labeling 
 $x_1,\ldots,x_{2g-2}$ of the vertices and a reference labeling $q_1,\ldots,q_{3g-3}$ of the edges of $\Gamma$, as in Definition \ref{def-int} for Feynman graphs.
We then consider the \emph{labeled tropical cover} $\hat\pi:C\rightarrow E$, where the source curve $C$ is in addition equipped with the labeling. The important difference between tropical covers and labeled tropical covers is the definition of isomorphism: for a labeled tropical cover, we require an isomorphism to respect the labels. As usual, we consider labeled tropical covers only up to isomorphism.

The combinatorial type of a labeled tropical cover is the combinatorial type of the source curve together with the labels, i.e., a Feynman graph.
\end{definition}

\begin{example}\label{exmp-labelledCover}
Figure \ref{fig-labelledCover} shows a labeled cover of degree $4$. The edges
labeled $q_2,q_3$ and $q_6$ are supposed to have weight $1$, edge $q_1$
and $q_4$ weight $2$ and $q_5$ weight $3$. The underlying Feynman graph is the one of Figure \ref{fig-raupe}.
\begin{figure}
 \begin{center}
  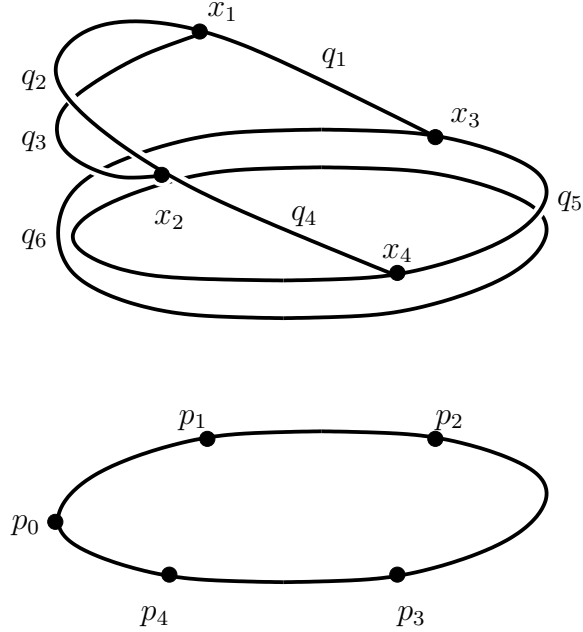
 \end{center}
\caption{A labeled tropical cover of $E$.}\label{fig-labelledCover}
\end{figure}

\end{example}

The definition of the generating series $F_g(q)$ for Hurwitz numbers has to be refined accordingly:
\begin{definition}\label{def-labelednumbercover}
 We fix once and for all a base point $p_0$ in $E$. 
For a tuple $\underline{a}=(a_1,\ldots,a_{3g-3})$, we define
$N_{\underline{a},g}^{\trop}$ to be the weighted number of labeled tropical covers $\hat\pi:C\rightarrow E$ of degree $\sum_{i=1}^{3g-3}a_i$, where $C$ has genus $g$, having their branch points at the prescribed positions and satisfying
$$ \#(\hat\pi^{-1}(p_0)\cap q_i)\cdot w_i=a_i$$
for all $i=1,\ldots,3g-3$. Each labeled tropical cover is counted with multiplicity $\prod_{i=1}^{3g-3} w_i$. Here, $w_i$ denotes the weight of the edge $q_i$. We call $\underline{a}$ the \emph{branch type} of the tropical cover at $p_0$.

We also define for a Feynman graph $\Gamma$ the number $N_{\underline{a}, \Gamma}^{\trop}$ to be the weighted number of labeled tropical covers as above with source curve of type $\Gamma$.

We set 
$$F_{g}(q_1,\ldots,q_{3g-3})= \sum_{\underline{a}}N_{\underline{a},g}^{\trop} q^{2\cdot\underline{a}}.$$
Here, the sum goes over all $\underline{a}\in \NN^{3g-3}$ and $q^{2\cdot \underline{a}}$ denotes the multi-index power $q^{2\cdot \underline{a}}= q_1^{2a_1}\cdot \ldots \cdot q_{3g-3}^{2a_{3g-3}}$.

\end{definition}

\begin{example}\label{exmp-N_aContribution}
Choose for example $g=3$ and $\underline{a}=(0,2,1,0,0,1)$. Let
$\Gamma$ be the Feynman graph depicted in Figure \ref{fig-raupe}.
\begin{figure}
\begin{center}
 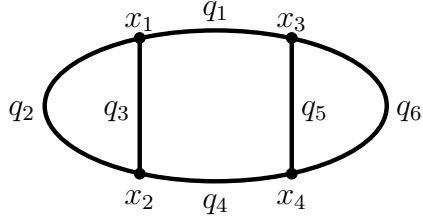
\end{center}
\caption{The Feynman graph $\Gamma$.}\label{fig-raupe}
\end{figure}
Then $N_{\underline{a}, \Gamma}^{\trop}=256$. All labeled covers contributing to $N_{\underline{a}, \Gamma}^{\trop}$
together with their multiplicities are depicted in Figure
\ref{fig-N_aContribution}. The number next to each label $q_i$ stands for the weight
of the labeled edge. The white dots are the points in the fiber of $p_0$ under
$\pi$.
\begin{figure}
 \begin{center}
 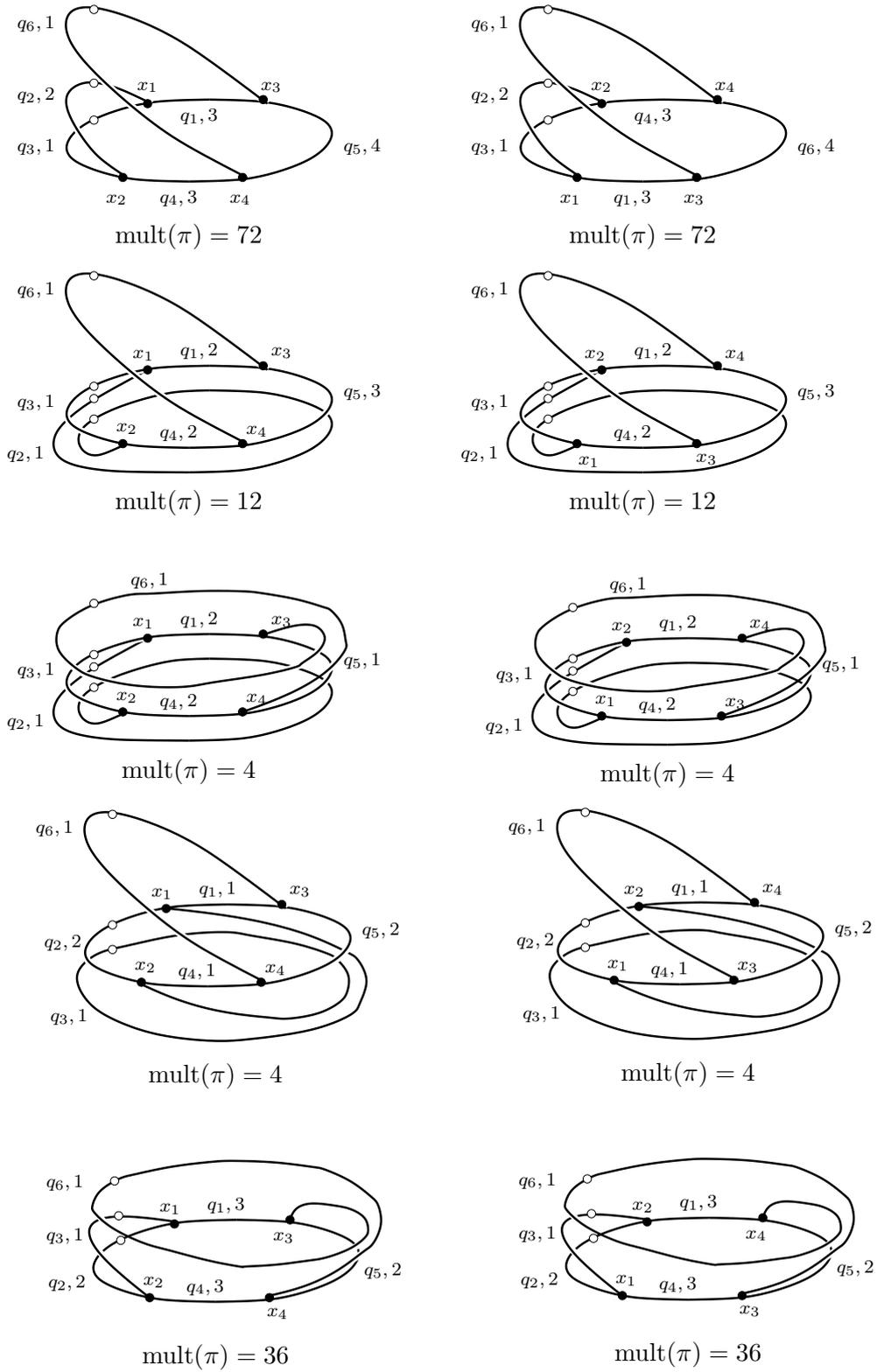
 \end{center}
\caption{All labeled tropical covers contributing to
$N_{(0,2,1,0,0,1),\Gamma}^{\trop}= 256$.}\label{fig-N_aContribution}
\end{figure}

\end{example}

Similarly, we refine the definition of the integrals of Definition \ref{def-int}:
\begin{definition}\label{def-labeledint}
 Let $\Gamma$ be a Feynman graph. As usual, the vertices are labeled with $x_i$ and the edges with $q_i$.
For the edge $q_k$ of $\Gamma$ connecting the vertices $x_i$ and $x_j$, we
change the definition of the integrand to 
$$P_k:=P(z_i-z_j,q_k),$$ 
where $P$ denotes the propagator of Definition \ref{def-prop}.
For a total ordering $\Omega$ of the vertices we then define the integral $$I_{\Gamma,\Omega}(q_1,\ldots,q_{3g-3}):= \int_{z_j\in \gamma_j} \prod_1^{3g-3} \left(-P_k\right)$$ just as in Definition \ref{def-int}.

We also set $$I_\Gamma(q_1,\ldots,q_{3g-3})= \sum_{\Omega} I_{\Gamma,\Omega}(q_1,\ldots,q_{3g-3}),$$ where the sum goes over all $(2g-2)!$ orders of the vertices. 
\end{definition}

We can now present the main result of our paper, the tropical mirror symmetry theorem for elliptic curves in its refined version:

\begin{theorem}[Tropical Mirror Symmetry for elliptic curves]\label{thm-refined}
Let $g>1$. For the definition of the invariants, see Definitions \ref{def-labeledtropcover} and \ref{def-labeledint}.
We have $$F_{g}(q_1,\ldots,q_{3g-3})= \sum_{\underline{a}}N_{\underline{a},g}^{\trop} q^{2\cdot \underline{a}}=\sum_{\Gamma}  I_\Gamma(q_1,\ldots,q_{3g-3}).$$

More precisely, the coefficient of the monomial $q^{2\cdot \underline{a}}$ in $I_\Gamma(q_1,\ldots,q_{3g-3})$ equals $N_{\underline{a}, \Gamma}^{\trop}$. 
\end{theorem}

The proof or Theorem \ref{thm-refined} follows immediatly from Theorem \ref{thm-bij}.

Note that the Tropical Mirror Symmetry Theorem naturally gives an interpretation of the Hurwitz number generating function $F_{g}(q_1,\ldots,q_{3g-3})$ in terms of a sum over Feynman graphs --- we can write it as
$$F_{g}(q_1,\ldots,q_{3g-3}) =\sum_\Gamma \left(\sum_{\underline{a}}N_{\underline{a},\Gamma}^{\trop} q^{2\cdot \underline{a}}\right),$$
and Theorem \ref{thm-refined} implies that the equality to the Feynman integral holds on the level of the summands for each graph. The same is true of course after setting $q_k=q$ for all $k$, thus going back to (unlabeled) tropical covers and (unrefined) Feynman integrals.
This is particularly interesting since all statements that hold on the level of graphs (such as the quasimodularity shown in Section \ref{sec-quasimod}) 
now become meaningful on the A-model side (i.e., for the generating function of Hurwitz numbers.)

 We now show how one can deduce Theorem \ref{thm-tropmirror} from this more refined version:
\begin{proof}[Proof of Theorem \ref{thm-tropmirror} using Theorem \ref{thm-refined}]
 For a fixed graph $\Gamma$, let $N_{d,\Gamma}^{\trop}$ be the number of (unlabeled) tropical covers of degree $d$ as in Definition \ref{def-trophurwitz}, where the combinatorial type of the source curve is $\Gamma$. As in Definition \ref{def-mult}, each cover $\pi:C\rightarrow E$ is counted with multiplicity $\frac{1}{|\Aut(\pi)|}\prod_{e}w_e$, where the product goes over all edges $e$ of $\Gamma$ and $w_e$ denotes the weight of the edge $e$.
As usual for Feynman graphs, we fix a reference labeling $x_i$ of the vertices and $q_i$ of the edges (see Definition \ref{def-int}). There exists a forgetful map $\ft$ from the set of labeled tropical covers satisfying the ramification conditions to the set of unlabeled covers by just forgetting the labels. We would like to study the cardinality of the fibers of $\ft$. Let $\pi:C\rightarrow E$ be an (unlabeled) cover such that the combinatorial type of $C$ is $\Gamma$. The automorphism group of $\Gamma$, $\Aut(\Gamma)$, acts transitively on the fiber $\ft^{-1}(\pi)$ by relabeling vertices and edges. So, to determine the cardinality of the set $\ft^{-1}(\pi)$, we think of it as the orbit under this action and obtain $|\ft^{-1}(\pi)|=\frac{|\Aut(\Gamma)|}{|\Aut(\pi)|}$, since the stabilizer of the action equals the set of automorphisms of $\pi$.
Each labeled cover in the set $\ft^{-1}(\pi)$ is counted with the same multiplicity $\prod_e w_e$, where the product goes over all edges $e$ in $\Gamma$.

Thus we obtain 

\begin{align*}&\sum_{\underline{a}|\sum a_i=d} N_{\underline{a},\Gamma}^{\trop}
= \sum_{\hat \pi:C\rightarrow E} \prod_e w_e = \sum_{\pi:C\rightarrow
E}\;\;\;\;\;\;\; \sum_{\hat\pi:C\rightarrow E|\ft(\hat\pi)=\pi} \prod_e w_e 
\\&= \sum_{\pi:C\rightarrow E} \frac{|\Aut(\Gamma)|}{|\Aut(\pi)|}\prod_e w_e  = |\Aut(\Gamma)|\cdot N_{d,\Gamma}^{\trop},\end{align*}
where the second sum goes over all labeled covers $\hat\pi:C\rightarrow E$ of degree $d$ and genus $g$ satisfying the conditions and such that the combinatorial type of $C$ is $\Gamma$, the third sum goes analogously over all (unlabeled) covers $\pi:C\rightarrow E$ and over the labeled covers in the fiber of the forgetful map, the third equality holds true because of the orbit argument we just gave and the last equality since an (unlabeled) cover is counted with multiplicity $\frac{1}{|\Aut(\pi)|}\prod_e w_e$.

We conclude 

\begin{align*}  &\sum_d N_{d,g}^{\trop} q^{2d} = \sum_d \sum_\Gamma N_{d,\Gamma}^{\trop} q^{2d}= \sum_d \sum_\Gamma \frac{1}{|\Aut(\Gamma)|} \sum_{\underline{a}|\sum a_i=d} N_{\underline{a},\Gamma}^{\trop} q^{2d}\\&
= \sum_\Gamma \frac{1}{|\Aut(\Gamma)|} \sum_d  \sum_{\underline{a}|\sum a_i=d} N_{\underline{a},\Gamma}^{\trop} q^{2d}. \end{align*}

Now we can replace $N_{\underline{a},\Gamma}^{\trop}$ by the coefficient of $q^{2\cdot\underline{a}}$ in $I_\Gamma(q_1,\ldots,q_{3g-3})$ by Theorem \ref{thm-refined}.
If we insert $q_k=q$ for all $k=1,\ldots,3g-3$ in $I_\Gamma(q_1,\ldots,q_{3g-3})$ we can conclude that the coefficient of $q^{2d}$ in $I_\Gamma(q)$ equals $\sum_{\underline{a}|\sum a_i=d}N_{\underline{a},\Gamma}^{\trop}$. Thus the above expression equals

$$\sum_\Gamma \frac{1}{|\Aut(\Gamma)|} I_\Gamma(q),$$

and Theorem \ref{thm-tropmirror} is proved.
\end{proof}

\begin{remark}\label{rem-basepoint}
 From a computational point of view, it makes sense to introduce a base point to the computations above.
In terms of labeled tropical covers, we then count covers as above which satisfy in addition the requirement that a fixed vertex, say e.g.\ $x_1$, is mapped to a fixed base point $p$. In terms of integrals, we set the variable $z_1=0$. The analogous statement to Theorem \ref{thm-refined} relating numbers of labeled tropical covers sending $x_1$ to $p$ to coefficients of integrals where we set $z_1=0$ holds true and can be proved along the same lines as the proof of Theorem \ref{thm-refined} presented here.
\end{remark}

\subsection{The propagator}\label{sec-prop}
In this subsection, we study the combinatorics of Feynman integrals. We show that the computation of a Feynman integral can be boiled down to a combinatorial hunt of monomials in a big generating function. We also express the integrals in terms of a constant coefficient of a multi-variate series, an expression that will become important in Section \ref{sec-quasimod} proving the quasimodularity of Feynman integrals.
In order to compute the integrals of Definition \ref{def-labeledint}, it is helpful to make a change of variables $x_j=e^{i\pi z_j}$ for each $j=1,\ldots,2g-2$.
Under this change of variables, each integration path $\gamma_j$ goes to (half) a circle around the origin. The integral is then nothing else but the computation of residues.
We start by giving a nicer expression of the propagator after the change of variables:
\begin{theorem}[The propagator]\label{thm-prop}

 The propagator $P(x,q)$ of Definition \ref{def-prop} with $x=e^{i\pi z}$ equals
\begin{equation}\label{PTaylor}
 P(x,q)=- \frac{x^2}{(x^2-1)^2}-\sum_{n=1}^\infty \left(\sum_{d|n}d \left(x^{2d}+ x^{-2d}\right)\right)q^{2n}.
\end{equation}

\end{theorem}

\begin{proof}

The claim of the theorem follows by comparing Taylor coefficients of both sides. To be more precise (see \cite{KK07})
$$
\wp (z,q) = z^{-2} + \sum_{k=2}^{\infty}(2k-1)G_{2k}(q)z^{2k-2},
$$
where $G_{2k}(q)$ is the classical weight $2k$ Eisenstein series, normalized to have a Fourier expansion of the shape
$$
G_{2k}(q)= 2 \zeta(2k) + \frac{2(2\pi i)^{2k}}{(2k-1)!}\sum_{n=1}^{\infty}\sigma_{2k-1}(n)q^{2n},
$$
where $\zeta(s)$ denotes the Riemann zeta function and $\sigma_{\ell}(n):= \sum_{d | n}d^{\ell}$ is the $\ell$th divisor sum. Note that at even integers the $\zeta$-funtion may be written in terms of Bernoulli numbers, defined via its generating functions
\begin{equation}\label{Bergen}
 \frac{t}{e^t -1} = \sum_{m=0}^{\infty} \frac{B_m}{m!}t^m.
\end{equation}
To be more precise, we have
$$
\zeta(2k)= (-1)^{k+1}\frac{(2\pi)^{2k}}{2(2k)!}B_{2k}.
$$
We next determine the Taylor expansion of the right-hand side of (\ref{PTaylor}). Differentiating (\ref{Bergen}) gives that
$$
-\frac{e^{2\pi i z}}{ (e^{2\pi i z} - 1)^2} = \frac{1}{(2\pi z)^2} + \sum_{m=1}^{\infty} \frac{m-1}{m!} B_m (2 \pi i z)^{m-2}.
$$
Moreover, in the second term we use the series expansion of the exponential function to obtain
\begin{align*}
-\sum_{n=1}^{\infty}\sum_{d | n}d\left(x^{2d} + x^{-2d}\right)q^{2n} & = - \sum_{n=1}^{\infty}\sum_{d | n}d\sum_{\ell=0}^{\infty}\left(1 + (-1)^{\ell}\right)\frac{(2\pi i d z)^{\ell}}{\ell!}q^{2n} \\ 
& = -2 \sum_{n=1}^{\infty}\sum_{d | n}d \sum_{\ell=0}^{\infty}\frac{(2\pi i d z)^{2\ell}}{(2\ell)!}q^{2n} \\
& = -2 \sum_{n=1}^{\infty} \sum_{\ell=0}^{\infty}\frac{(2\pi i z)^{2\ell}}{(2\ell)!}\sum_{d | n}d^{2\ell +1}q^{2n} \\
& = -2 \sum_{n=1}^{\infty} \left(\sum_{\ell=0}^{\infty}\frac{(2\pi i z)^{2\ell}}{(2\ell)!}\sigma_{2\ell+1}(n) \right) q^{2n}. \\
\end{align*}
Now the claim follows by comparing coefficients of $z^{2\ell}$.

\end{proof}

Now let us go back to a fixed Feynman graph $\Gamma$ and consider the function we have to integrate.
Denote the vertices of the edge $q_k$ of $\Gamma$ by $x_{k_1}$ and $x_{k_2}$. 
Since the derivative of $\ln(x_i)$ is $\frac{1}{x_i}$, after the coordinate change we integrate the function $$P_\Gamma:=\prod_{k=1}^{3g-3} \left(-P\Big(\frac{x_{k_1}}{x_{k_2}},q_k\Big)\right) \cdot \frac{1}{i\pi x_1\cdot\ldots\cdot i\pi  x_{2g-2}}.$$

After the coordinate change, the integration paths are half-circles around the origin. Since our function is symmetric (there are only even powers of $x$), we can compute this integral as $\frac{1}{2}$ times the integral of the same function along a whole circle. Since the function has only one pole at zero (within the range of integration), we can compute the integral along the whole circle as $2i\pi$ times the residue at zero by the Residue Theorem.

It follows that the integral equals the constant coefficient of
\begin{equation}P'_\Gamma:=\prod_{k=1}^{3g-3} \left(-P\Big(\frac{x_{k_1}}{x_{k_2}},q_k\Big)\right).\label{eq-integrand}\end{equation}

Note that it follows from Theorem \ref{thm-prop} that $-P\big(\frac{x}{y},q\big)=- P\big(\frac{y}{x},q\big)$. 
This is obvious for the (Laurent-polynomial) coefficients of $q^d$ with $d>0$.
For the constant coefficient, it follows since
\begin{equation} \frac{\big(\frac{x}{y}\big)^2}{\left( \big(\frac{x}{y}\big)^2-1\right)^2} = \frac{x^2y^2}{(x^2-y^2)^2} = \frac{x^2y^2}{(y^2-x^2)^2}= \frac{\big(\frac{y}{x}\big)^2}{\left( \big(\frac{y}{x}\big)^2-1\right)^2}.\label{eq-constterm}\end{equation}
Therefore it is not important which vertex of $q_k$ we call $x_{k_1}$ and which $x_{k_2}$ (this also explains the independence of the sign $z_i-z_j$ resp.\ $z_j-z_i$ in Definition \ref{def-int}).
To compute the (in the $x_k$) constant coefficient of $P'_\Gamma$, we have to express the (in $q_k$) constant coefficient of each factor, $$\frac{\left(\frac{x_{k_1}}{x_{k_2}}\right)^2}{\left( \left(\frac{x_{k_1}}{x_{k_2}}\right)^2-1\right)^2},$$ as a series.
Depending on whether $\left|\frac{x_{k_1}}{x_{k_2}}\right|<1$ or $\left|\frac{x_{k_2}}{x_{k_1}}\right|<1$, we can use the left or the right expression of Equation (\ref{eq-constterm}) for the constant coefficient and expand the denominator as product of geometric series.
The following lemma shows how to expand the constant coefficient as a series, depending on the absolute value of the ratio of the two involved variables. This explains why different orders $\Omega$ can produce different integrals $I_{\Gamma,\Omega} (q_1,\ldots,q_{3g-3})$: the position of the integration paths determine the series expansion of the constant coefficients.

\begin{lemma}\label{lem-constterm}
 Assume $|x|<1$. Then

$$\frac{x^2}{\left( x^2-1\right)^2} =\sum_{w=1}^\infty w\cdot x^{2w}.$$
\end{lemma}
The proof follows easily after expanding the factors as geometric series.

The discussion of this subsection can be summed up as follows:

\begin{lemma}\label{lem-seriesconstterm}

Fix a Feynman graph $\Gamma$ and an order $\Omega$ as in Definition \ref{def-int}, and a tuple $(a_1,\ldots,a_{3g-3})$ as in Definition \ref{def-labelednumbercover}.
We express the coefficient of $q^{2\cdot\underline{a}}$ in $I_{\Gamma,\Omega}(q_1,\ldots,q_{3g-3})$ of Definition \ref{def-labeledint}.
Assume $k$ is such that the entry $a_k=0$, and assume the edge $q_k$ connects the two vertices $x_{k_1}$ and $x_{k_2}$. Choose the notation of the two vertices $x_{k_1}$ and $x_{k_2}$ such that the chosen order $\Omega$ implies $\left|\frac{x_{k_1}}{x_{k_2}}\right|<1$ for the starting points on the integration paths.
Then the coefficient of $q^{2\cdot\underline{a}}$ in $I_{\Gamma,\Omega}(q_1,\ldots,q_{3g-3})$ equals the constant term of the series 
\begin{equation}\prod_{k|a_k=0} \left( \sum_{w_k=1}^\infty w_k\cdot\left(\frac{x_{k_1}}{x_{k_2}}\right)^{2w_k}\right)
\cdot \prod_{k|a_k\neq 0} \left(\sum_{w_k|a_k}w_k \left(\left(\frac{x_{k_1}}{x_{k_2}}\right)^{2w_k}+\left(\frac{x_{k_2}}{x_{k_1}}\right)^{2w_k}\right)\right).
\label{eq-series}\end{equation}
 
\end{lemma}

The discussion above will also be important in Section \ref{sec-quasimod} for proving the quasimodularity of Feynman integrals. We sum it up in a way suitable for this purpose. We insert $q_k=q$ for each $k$ coming back to the integrals of Definition \ref{def-int}. 
For a Taylor series
\[
F\left(x_1, \dots, x_n\right)=\sum\limits_{1\leq j\leq n\atop{a_j\geq 0}}\alpha\left(a_1, \dots, a_n\right) x_1^{a_1}\cdot\dots\cdot x_n^{a_n}
\]
we denote
\[
\text{coeff}_{\left[x_1^{a_1},\dots, x_n^{a_n}\right]}(F)=\alpha\left(a_1, \dots, a_n\right).
\]
Moreover, for a fixed Feynman graph $\Gamma$ and an order $\Omega$ as in Definition \ref{def-int}, let 
\[
P_{\Gamma, \Omega}:=\prod_{k=1}^{3g-3}\left(-P\left(\frac{x_{k_1}}{x_{k_2}}, q\right)\right),
\]
where $P$ as usual denotes the propagator of Definition \ref{def-prop} resp. Theorem \ref{thm-prop}.
\begin{lemma}\label{lem-Igammaconst}
 For $\Gamma$ and $\Omega$ as in Definition \ref{def-int}, we have
\[
I_{\Gamma, \Omega}(q)=\text{coeff}_{\left[x_1^0,\dots, x_{2g-2}^0\right]}\left(P_{\Gamma, \Omega}\right).
\]
\end{lemma}
The claim follows immediatly from Equation (\ref{eq-integrand}).

\subsection{The bijection}\label{subsec-bij}
For a fixed Feynman graph $\Gamma$ and tuple $(a_1,\ldots,a_{3g-3})$, we are now ready to directly relate nonzero contributions to the constant term of the series given in (\ref{eq-series}) for each order $\Omega$ to tropical covers contributing to $N_{\underline{a},\Gamma}^{\trop}$, thus proving Theorem \ref{thm-refined}.

We express the constant term as a sum over products containing one term of each factor of the series in (\ref{eq-series}):

\begin{definition}\label{def-tuples}
Fix $\Gamma$, $\Omega$ and $(a_1,\ldots,a_{3g-3})$ as in Definition \ref{def-int} resp. \ref{def-labelednumbercover}.
Consider a tuple of powers $a_k$ and terms of the series in (\ref{eq-series}) $$\big((a_k,T_k)\big)_{k=1,\ldots,3g-3}= \Bigg(\Big(a_k, w_k\cdot\Big(\frac{x_{k_i}}{x_{k_j}}\Big)^{2w_k}\Big)\Bigg)_{k=1,\ldots,3g-3},$$ where $i=1$ and $j=2$ if $a_k=0$, and $\{i,j\}=\{1,2\}$ otherwise.
We require the product of the terms, $\prod_{k=1}^{3g-3} T_k$, to be constant in each $x_i$, $i=1,\ldots,2g-2$.

We denote the set of all such tuples by $\mathcal{T}_{\underline{a},\Gamma,\Omega}$. 
\end{definition}
Obviously, each tuple yields a summand of the constant term of the series in (\ref{eq-series}) (and thus, by Lemma \ref{lem-seriesconstterm}, a contribution to the $q^{2\cdot \underline{a}}$-coefficient of $I_{\Gamma,\Omega}(q_1,\ldots,q_{3g-3})$), and vice versa, each summand arises from a tuple in $\mathcal{T}_{\underline{a},\Gamma,\Omega}$.
Note that the contribution of a tuple equals \begin{equation}\prod_{k=1}^{3g-3} T_k= \prod_{k=1}^{3g-3}w_k\label{eq-contribution}.\end{equation}

\begin{definition}\label{def-tropcoverorder}
 Let $\hat\pi:C\rightarrow E$ be a labeled tropical cover as in Definition \ref{def-labeledtropcover} contributing to $N_{\underline{a},\Gamma}^{\trop}$.
We cut $E$ at the base point $p_0$ and flatten it to an interval following clockwise orientation. We define an order of the vertices $x_i$ of $C$ given by the natural order of their image points on the interval.
For a given order $\Omega$ as in Definition \ref{def-int}, we let $N_{\underline{a},\Gamma,\Omega}^{\trop}$ be the weighted number of labeled tropical covers as in \ref{def-labeledtropcover} (i.e., of degree $\sum_{i=1}^{3g-3}a_i$, where the source curve has combinatorial type $\Gamma$, having their branch points at the prescribed positions and satisfying
$ \#(\hat\pi^{-1}(p_0)\cap q_i)\cdot w_i=a_i$ for all $i=1,\ldots,3g-3$, where $w_i$ denotes the weight of the edge $q_i$), and in addition satisfying that the above order equals $\Omega$. As usual, each cover is counted with multiplicity $\prod_{k=1}^{3g-3}w_k$.
\end{definition}
Note that obviously we have $\sum_\Omega N_{\underline{a},\Gamma,\Omega}^{\trop}= N_{\underline{a},\Gamma}^{\trop}$, where the sum goes over all $(2g-2)!$ orders $\Omega$ of the vertices.

\begin{example}
The cover in Figure \ref{fig-labelledCover} cut at $p_0$ and
flattened to an interval is depicted in Figure \ref{fig-vertexOrder}. Its vertex
ordering $\Omega$ is given by $x_1<x_2<x_3<x_4$.

\begin{figure}
 \begin{center}
  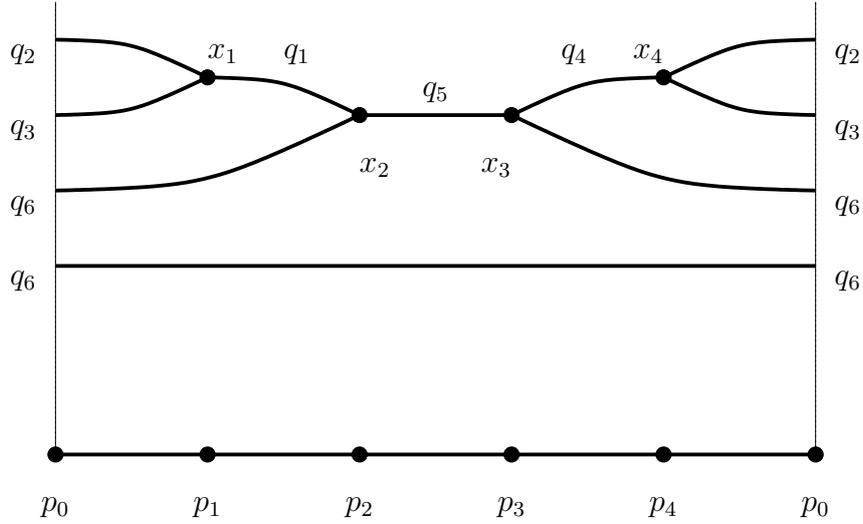
 \end{center}
\caption{The cover of Figure \ref{fig-labelledCover} cut at
$p_0$.}\label{fig-vertexOrder}
\end{figure}
 \end{example}
\begin{example}
 Go back to Example \ref{exmp-N_aContribution} where we determined $ N_{\underline{a},\Gamma}^{\trop}$ for the Feynman graph $\Gamma$ in Figure \ref{fig-raupe} and $\underline{a}=(0,2,1,0,0,1)$.
Note that there are two orders $\Omega$ that yield nonzero contributions, namely $x_1<x_3<x_4<x_2$ (call it $\Omega_1$) and $x_2<x_4<x_3<x_1$ (call it $\Omega_2$). 
In Figure \ref{fig-N_aContribution}, the covers with order $\Omega_1$ appear in the left column, the covers with $\Omega_2$ in the right column.
For both orders, we have $ N_{\underline{a},\Gamma,\Omega_i}^{\trop}=128$, $i=1,2$. Altogether, we have  $ N_{\underline{a},\Gamma}^{\trop}= N_{\underline{a},\Gamma,\Omega_1}^{\trop}+N_{\underline{a},\Gamma,\Omega_2}^{\trop}=128+128=256$.
\end{example}

\begin{theorem}\label{thm-bij}
Fix a Feynman graph $\Gamma$, an order $\Omega$ and a tuple $(a_1,\ldots,a_{3g-3})$ as in Definition \ref{def-int} resp. \ref{def-labelednumbercover}.

There is a bijection between the set of labeled tropical covers contributing to $N_{\underline{a},\Gamma,\Omega}^{\trop}$ and the set $\mathcal{T}_{\underline{a},\Gamma,\Omega}$ of tuples contributing to the $q^{2\cdot \underline{a}}$-coefficient of $I_{\Gamma,\Omega}(q_1,\ldots,q_{3g-3})$ (see Definition \ref{def-tropcoverorder} and \ref{def-tuples}).

The bijection identifies the coefficients $w_k$ of the terms $T_k$ of a tuple with the weights of the edges of the corresponding labeled tropical cover. 
In particular, the contribution of a tuple to the coefficient of $q^{2\cdot \underline{a}}$ in $I_{\Gamma,\Omega}(q_1,\ldots,q_{3g-3})$ equals the multiplicity of the corresponding labeled tropical cover, with which it contributes to $N_{\underline{a},\Gamma,\Omega}^{\trop}$.
\end{theorem}
Note that it follows immediatly from Theorem \ref{thm-bij} that $N_{\underline{a},\Gamma,\Omega}^{\trop}$ equals the coefficient of $q^{2\cdot \underline{a}}$ in $I_{\Gamma,\Omega}(q_1,\ldots,q_{3g-3})$, hence 
the coefficient of $q^{2\cdot \underline{a}}$ in $I_\Gamma(q_1,\ldots,q_{3g-3})$ equals $N_{\underline{a}, \Gamma}^{\trop}$ and Theorem \ref{thm-refined} is proved.

To prove Theorem \ref{thm-bij}, we set up the map sending a tuple in $\mathcal{T}_{\underline{a},\Gamma,\Omega}$ to a tropical cover in Construction \ref{const-bij}. The theorem then follows from Lemma \ref{lem-themap} stating that this construction indeed yields a map as required and Lemma \ref{lem-inverse} stating that it has a natural inverse and therefore is a bijection.

As usual, let $\Gamma$, $\Omega$ and $\underline{a}$ be fixed as in Definition \ref{def-int} resp.\ \ref{def-labelednumbercover}.
\begin{construction}\label{const-bij}
Draw an interval from $p_0$ to $p_0'$ that can later be glued to $E$ by identifying $p_0$ and $p_0'$, and a rectangular box above in which we can step by step draw a cover of $E$ following the construction (see Figure \ref{fig-box}). The vertical sides of the box are called $L$ and $L'$ and represent points which (if they belong to the cover after the construction) are pairwise identified and mapped to the base point.

\begin{figure}
\begin{center}
 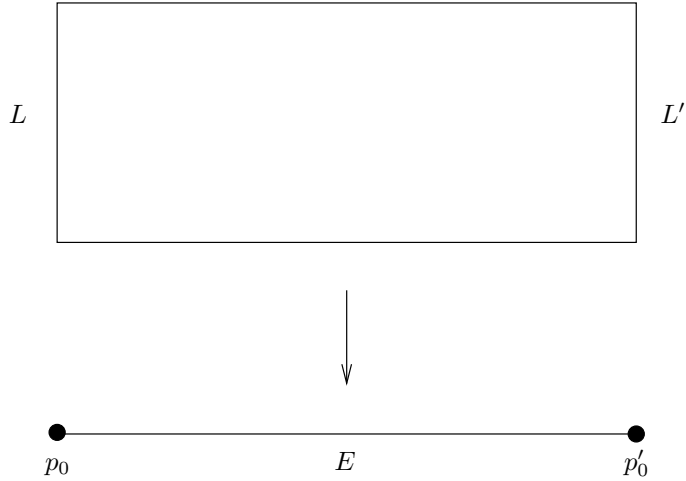
\end{center}
\caption{Preparation to construct a tropical cover from a tuple.}
\label{fig-box}
\end{figure}

Given a tuple $$\big((a_k,T_k)\big)_{k=1,\ldots,3g-3}= \Bigg(\Big(a_k, w_k\cdot\Big(\frac{x_{k_i}}{x_{k_j}}\Big)^{2w_k}\Big)\Bigg)_{k=1,\ldots,3g-3}$$ of $\mathcal{T}_{\underline{a},\Gamma,\Omega}$,
\begin{itemize}
 \item draw dots labeled $x_1,\ldots,x_{3g-3}$ into the box, from left to right (and slightly downwards, to keep some space to continue the picture), as determined by the order $\Omega$, one dot above each point condition $p_i\in E$ where we fix the branch points;
\item for the term $T_k= w_k\cdot\Big(\frac{x_{k_i}}{x_{k_j}}\Big)^{2w_k}$ draw an edge leaving vertex $x_{k_i}$ to the right and entering vertex $x_{k_j}$ from the left --- if $a_k=0$, let this edge be a straight line connecting these two vertices, if $a_k\neq 0$ let it first leave the box at $L'$ and enter again at $L$, altogether $\frac{a_k}{w_k}$ times, before it enters $x_{k_j}$;
\item give the edges drawn in the previous item weight $w_k$. As always, the lengths of the edges are then determined by the differences of the image points of the $x_i$ and the weights.
\end{itemize}
Glue the corresponding points on $L$ and $L'$ to obtain a cover of $E$.
\end{construction}

\begin{example}\label{ex-const}

Let $\Gamma$ be the Feynman graph of Figure \ref{fig-raupe}
and let $\underline{a}=(0,2,2,0,1,0)$. Moreover, choose the ordering
$x_1<x_3<x_4<x_2$ and pick the terms $T_1=\left(\frac{x_1}{x_3}\right)^2$, 
$T_2=2\cdot\left(\frac{x_2}{x_1}\right)^{2\cdot 2}$, $T_3=\left(\frac{x_1}{x_2}\right)^2$,
$T_4=\left(\frac{x_4}{x_2}\right)^2$, $T_5=\left(\frac{x_4}{x_3}\right)^2$ and
$T_6=2\cdot\left(\frac{x_3}{x_4}\right)^{2\cdot 2}$ from the series in (\ref{eq-series}). When applying
Construction \ref{const-bij} we obtain the picture shown in Figure \ref{fig-bij} before gluing.
\begin{figure}
 \begin{center}
  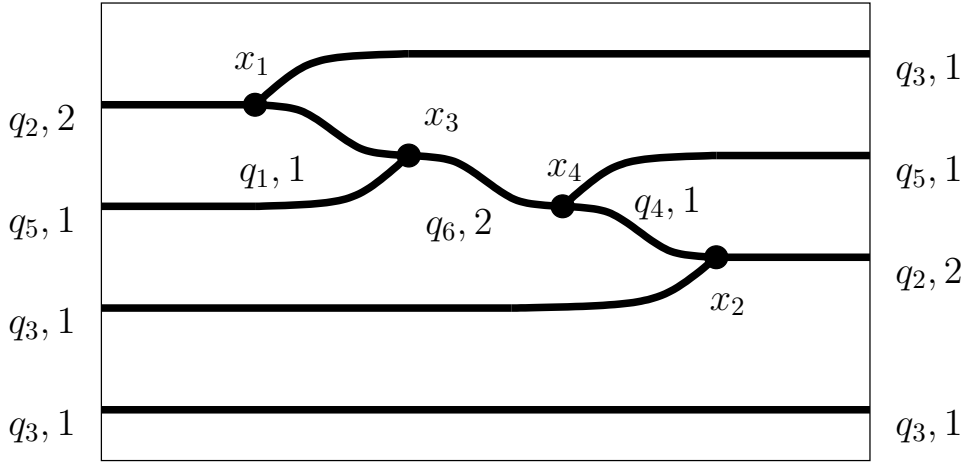
 \end{center}
\caption{Applying Construction \ref{const-bij} in Example \ref{ex-const}.}\label{fig-bij}
\end{figure}
\end{example}

\begin{lemma}\label{lem-themap}
 Construction \ref{const-bij} defines a map from $\mathcal{T}_{\underline{a},\Gamma,\Omega}$ to the set of tropical covers contributing to $N_{\underline{a},\Gamma,\Omega}^{\trop}$ (see Definition \ref{def-tropcoverorder} and \ref{def-tuples}).
\end{lemma}
\begin{proof}
 Since the integrand of Definition \ref{def-labeledint} is set up such that we have a term containing a power of $\frac{x_i}{x_j}$ in the tuples in $\mathcal{T}_{\underline{a},\Gamma,\Omega}$ if and only if there is an edge $q_k$ connecting $x_i$ and $x_j$, it is clear that we produce a cover whose source curve has combinatorial type equal to the labeled Feynman graph $\Gamma$ (see also Equation (\ref{eq-integrand})).
It is also obvious that the vertices are mapped to the interval respecting the order $\Omega$.
To see that it is a tropical cover at all, we have to verify the balancing condition at each vertex $x_i$. This follows from the fact that we require the product of all terms to be constant in $x_i$: since $\Gamma$ is trivalent, we have three edges adjacent to $x_i$, to fix notation call them (without restriction) $q_1$, $q_2$, and $q_3$. Assume (also without restriction) that the other vertex of $q_j$, $j=1,\ldots,3$ is $x_j$.
The only three terms in the product $\prod_{k=1}^{3g-3}T_k$  involving $x_i$ are then
$$ w_1\cdot\Big(\frac{x_{1}}{x_{i}}\Big)^{2w_1}, w_2\cdot\Big(\frac{x_{i}}{x_{2}}\Big)^{2w_2} \mbox{, and } w_3\cdot\Big(\frac{x_{i}}{x_{3}}\Big)^{2w_3}, $$
where we picked an arbitrary choice between a quotient such as $\frac{x_1}{x_i}$ and its inverse in each term for now. This choice is again made without restriction, just to fix the notation in the terms $T_k$ of the given tuple. (Of course, in general, if some of the $a_j$,$j=1,\ldots,3$ are zero, the choice has to respect the order $\Omega$.)
In our fixed but arbitrary choice the requirement that the product is constant in $x_i$ translates to the equation
$$ -2w_1+2w_2+2w_3=0.$$
The construction implies that here, the edge $q_1$ enters $x_i$ from the left with weight $w_1$ while $q_2$ and $q_3$ leave the vertex $x_i$ to the right with weight $w_2$ and $w_3$, hence the balancing condition is fulfilled.
Since the direction of the edges we draw in the construction reflects the fact that the corresponding two vertices show up in the numerator resp.\ denominator of the quotient, it is obvious that the instance for which we fixed a notation generalizes to any situation; and the balancing condition at $x_i$ is always equivalent to the requirement that the product is constant in $x_i$.
It is clear from the construction that the labeled tropical cover we have built has its branch points at the required positions. If $a_k=0$, the construction implies that  $ \hat\pi^{-1}(p_0)\cap q_k=\emptyset$ and thus 
 $ \#(\hat\pi^{-1}(p_0)\cap q_k)\cdot w_k=0=a_k$ as required.
If $a_k\neq 0$, we draw $\frac{a_k}{w_k}$ points on $L$ resp.\ $L'$ that are identified to give $\frac{a_k}{w_k}$ preimages of $p_0$ in $q_k$, thus 
$ \#(\hat\pi^{-1}(p_0)\cap q_k)\cdot w_k=a_k$ holds in general. In particular, the degree of the cover is $\sum_{k=0}^{3g-3}a_k$.

Obviously, the map identifies the coefficient $w_k$ of a term $T_k$ with the weight of the edge $q_k$, therefore the contribution of a tuple to the $q^{2\cdot \underline{a}}$-coefficient of the integral $I_{\Gamma,\Omega}(q_1,\ldots,q_{3g-3})$ (given by Equation (\ref{eq-contribution})) equals the contribution of the corresponding tropical cover to $N_{\underline{a},\Gamma,\Omega}^{\trop}$ by Definition \ref{def-tropcoverorder}.
The statement follows.
\end{proof}

\begin{lemma}\label{lem-inverse}
 The map of Lemma \ref{lem-themap} has a natural inverse and is a bijection.
\end{lemma}
\begin{proof}
 One can reverse Construction \ref{const-bij} in the obvious way.
For any edge $q_k$ which does not pass $L$ resp.\ $L'$, we set $a_k=0$. 
Assume $q_k$ connects the two vertices $x_{k_1}$ and $x_{k_2}$ and assume the order $\Omega$ satisfies $x_{k_1}<x_{k_2}$. Then it follows from Definition \ref{def-int} that we have to pick integration paths satisfying $\left|\frac{x_{k_1}}{x_{k_2}}\right|<1$ and thus using Lemma \ref{lem-constterm} the power series expansion of the corresponding constant term of the propagator (see Theorem \ref{thm-prop}) then contains quotients $\frac{x_{k_1}}{x_{k_2}}$ as required.
We pick the term $w_k\cdot\Big( \frac{x_{k_1}}{x_{k_2}}\Big)^{2w_k}$ for our tuple.

For any edge $q_k$ that passes $l_k$ times with weight $w_k$ through $L$ resp.\ $L'$, we set $a_k=w_k\cdot l_k$. Obviously, $w_k$ is a divisor of $a_k$ and thus the term $ w_k\cdot \Bigg(\Big( \frac{x_{k_1}}{x_{k_2}}\Big)^{2w_k}+\Big( \frac{x_{k_2}}{x_{k_1}}\Big)^{2w_k}\Bigg)$ shows up in the $q_k^{2a_k}$-coefficient of the propagator as required (see Theorem \ref{thm-prop}) and we can pick the summand corresponding to the orientation of our arrow as term for our tuple.

It is obvious that this inverse construction produces a bijection.
\end{proof}

Remember that a bridge of a connected graph is an edge whose removal produces two connected components.
\begin{corollary}\label{cor-bridgegraph}
 A Feynman graph $\Gamma$ satisfies $N_{\underline{a},\Gamma}^{\trop}=0$ for every $\underline{a}$, or (by Theorem \ref{thm-refined}) equivalently, $I_{\Gamma}(q_1,\ldots,q_{3g-3})=0$ if and only if $\Gamma$ contains a bridge.
\end{corollary}
\begin{proof}
One can view a tropical cover as a system of rivers flowing into each other without any source or sink, since the weights of the edges are positive and the balancing condition is satisfied.
A graph with a bridge must have zero flow on the bridge, and thus cannot be the source of a tropical cover.

Alternatively, to get a nonzero contribution to the coefficient of  $q^{2\cdot \underline{a}}$ in the integral $I_{\Gamma}(q_1,\ldots,q_{3g-3})=0$, we must be given an order $\Omega$ and a tuple in $\mathcal{T}_{\underline{a},\Gamma,\Omega}$. But by Lemma \ref{lem-themap} such a tuple only exists if the balancing condition is satisfied at every vertex, thus the argument we just gave shows the coefficient is zero for a graph with a bridge.

Vice versa, we have to show that there exists a cover for every graph without a bridge.
To see this, we give an algorithm below how to construct for a given bridgeless graph an orientation of the edges that satisfies the following: there is no cut into two connected components $\Gamma_1$ and $\Gamma_2$ for which all cut edges are oriented from $\Gamma_1$ to $\Gamma_2$. It is easy to see that for such an orientation, we can insert positive weights for the edges such that the balancing condition is satisfied at every vertex (we just add enough water to the system of rivers). Thus the statement follows from Construction \ref{const-orientation} and Lemma \ref{lem-noorientedcut} below.
 
\end{proof}

\begin{construction}\label{const-orientation}
Let $\Gamma$ be a bridgeless graph.
\begin{enumerate}
 \item Choose an arbitrary cycle and orient its edges in one direction. Also choose a reference vertex $V$ on the cycle. Let $K$ denote the set of vertices on the cycle, this is the set of ``known vertices'' that we will enlarge in the following steps.
\item Let $U_1,\ldots,U_s$ denote the connected components of the subgraph induced on the vertex set of $\Gamma$ minus $K$. If $s\geq 1$, choose an arbitrary vertex $W\in U_1$. Since $\Gamma$ is connected, there is a path from $V$ to $W$ and we can choose it such that it respects our so far fixed orientations for the edges. At some point, the path must leave the ``known part'' and enter $U_1$, call this edge $E_1$. Since $E_1$ is not a bridge, there must be at least a second edge $E_2$ connecting the known part to $U_1$. We go along $E_1$ into $U_1$ until we reach $W$, and then continue until we hit via $E_2$ the known part again. We orient the edges we follow on the way. We add the set of vertices we meet on the way to $K$ and start again at (2).
\item At each step described above, we increase the vertex set of the known part. If all vertices are known, we orient the remaining edges arbitrarily.
\end{enumerate}
\end{construction}
\begin{lemma}\label{lem-noorientedcut}
Given a bridgeless graph $\Gamma$, we can use Construction \ref{const-orientation} to orient the edges such that the following is satisfied:
\begin{enumerate}
 \item Every vertex is contained in an oriented cycle that also contains the reference vertex $V$.
\item There is no cut into two connected components $\Gamma_1$ and $\Gamma_2$ such that all cut edges are oriented from $\Gamma_1$ to $\Gamma_2$.
\end{enumerate}
\end{lemma}
\begin{proof}
The first statement is obvious from the construction: we start with a cycle containing $V$ and add oriented ``handles''. For the second statement, assume there was such a cut, and assume without restriction that $V$ is in $\Gamma_1$. Choose an arbitrary vertex $W$ in $\Gamma_2$. By (1), there is an oriented cycle containing $W$ and $V$. This cycle must contain at least two cut edges which are thus oriented in opposite direction.
\end{proof}

\section{Quasimodularity}\label{sec-quasimod}
In this section, we review the quasimodularity of the individual Feynman integrals \cite{GM16}.
The Mirror Symmetry Theorem \ref{thm-mirror} implies
$$F_g(q)=\sum_{d}N_{d,g}q^{2d}= \sum_{\Gamma}\frac{1}{|\Aut(\Gamma)|}\Big(\sum_\Omega I_{\Gamma,\Omega}\Big),$$
where the sum runs over all Feynman graphs $\Gamma$ and orders $\Omega$ as in Definition \ref{def-int}.
It is known that $F_g(q)$ is a quasimodular form \cite{Dij95, KZ95}.
In fact, already the individual summands $I_{\Gamma,\Omega}(q)$ are quasimodular forms (of mixed weight). We start by recalling the necessary definitions.

\begin{definition}
A function $f:\H\to\C$ (where $\H$ denotes the complex upper half plane) is called an \emph{almost holomorphic modular form of weight $k\in\Z$} if
$\forall \left(\begin{smallmatrix} a&b\\c&d\end{smallmatrix}\right)\in \SL_2(\Z)$, we have
\[
f\left(\frac{a\tau+b}{c\tau+d}\right)=(c\tau+d)^k f(\tau)
\]
and $f$ has the shape $(\tau=u+iv)$
\begin{equation}\label{shape}
f(\tau)= \sum_{j=0}^r \frac{f_j(\tau)}{v^j}
\end{equation}
with $f_j$ holomorphic in $\H\cup \{\infty\}$. The constant term $f_0$ is called a \emph{quasimodular form} and $r$ is called the \emph{depth} of $f$.
Note that
\[
f_0(\tau)=\lim_{\overline{\tau}\to\infty} f(\tau),
\]
if we view $\tau$ and $\bar{\tau}$ as independent variables.
\end{definition}


The following theorem is proved in \cite{GM16}, Corollary 8.4:

\begin{theorem}\label{QuasiTheorem} For all Feynman graphs $\Gamma$ and orders $\Omega$ as in Definition \ref{def-int},
the function $I_{\Gamma, \Omega}$ is a quasimodular form of mixed weight. The highest appearing weight is $6g-6$.
\end{theorem}

\begin{remark}
If $\Gamma$ is of genus $g>1$, by Theorem \ref{QuasiTheorem}, $I_{\Gamma
}(q)=\sum_{\Omega}I_{\Gamma,\Omega}(q)$ is a quasi-modular form of mixed weight (at most
$6g-6$). Hence $I_{\Gamma}(q)\in\mathbb{Q}[E_{2},E_{4},E_{6}]$, where
$E_{2},E_{4},E_{6}$ are the Eisenstein series of weight $2$, $4$ and $6$.
Hence%
\[
I_{\Gamma}(q)=\sum_{\substack{2i+4j+6k\leq 6g-6\\i,j,k\in\mathbb{N}_{0}}%
}c_{\Gamma}^{i,j,k}\cdot E_{2}^{i}E_{4}^{j}E_{6}^{k}%
\]
with $c_{\Gamma}^{i,j,k}\in\mathbb{Q}$. Theorem \ref{thm-refined} gives an algorithm for computing the
series%
\[
I_{\Gamma}(q)=\sum\nolimits_{\underline{a}}N_{\underline{a},\Gamma}%
^{\trop}q^{2\left\vert \underline{a}\right\vert }\text{.}%
\]
Determining $I_{\Gamma}(q)$ up to sufficiently high order yields a linear system
of equations for the coefficients $c_{\Gamma}^{i,j,k}$ which admits a unique
solution. In this way, we can use our \textsc{Singular} package \textsc{ellipticcovers.lib} \cite{BBBM} to compute the respresentation of $I_{\Gamma}(q)$ in terms of Eisenstein series. A concrete example is presented in Example \ref{ex-modular}.
Note that we can compute more orders than necessary, thus yielding an overdetermined system of linear equations. Theorem \ref{QuasiTheorem} predicts that such an overdetermined system still has a unique solution. For our example below, we verified this statement computationally.
\end{remark}

\begin{remark}
Note that in an earlier version of this paper, we had claimed that the individual summands $I_{\Gamma, \Omega}(q)$ are quasimodular forms of weight $6g-6$. The proof had a gap however --- it holds true that they are quasimodular forms of mixed weight up to $6g-6$.
As it follows already from \cite{KZ95}, the whole sum $\sum_{\Gamma}\frac{1}{|\Aut(\Gamma)|}(\sum_\Omega I_{\Gamma,\Omega})$ is a quasimodular form of pure weight $6g-6$, so the contributions of lower weight cancel in the sum.
The example below shows that they may already cancel in a sum $I_\Gamma =\sum_\Omega  I_{\Gamma, \Omega}(q)$. It is an interesting question whether this holds true in general.
\end{remark}

\begin{example}\label{ex-modular}
For the Feynman graph $\Gamma_{1}$ as depicted in Figure \ref{fig-raupe}, we
compute%
\[
I_{\Gamma_{1}}(q)=32q^{4}+1792q^{6}+25344q^{8}+182272q^{10}+886656q^{12}%
+O(q^{14})\text{.}%
\]
This determines a unique solution for the coefficients $$c_{\Gamma_{1}}%
^{0,0,2},c_{\Gamma_{1}}^{0,3,0},c_{\Gamma_{1}}^{1,1,1},c_{\Gamma_{1}}%
^{2,2,0},c_{\Gamma_{1}}^{3,0,1},c_{\Gamma_{1}}^{4,1,0},c_{\Gamma_{1}}^{6,0,0}$$
of the weight $12$ monomials, which corresponds to the quasi-modular
representation%
\[
I_{\Gamma_{1}}(q)=\frac{16}{1492992}\left(  4 E_{6}^{2}+4 E_{4}%
^{3}-12 E_{2}E_{4}E_{6}-3 E_{2}^{2}E_{4}^{2}+4 E_{2}^{3}%
E_{6}+6 E_{2}^{4}E_{4}-3 E_{2}^{6}\right)  \text{.}%
\]
In the same way, for the graph $\Gamma_{2}$ as in Figure \ref{fig-stern} we
obtain%
\begin{figure}
 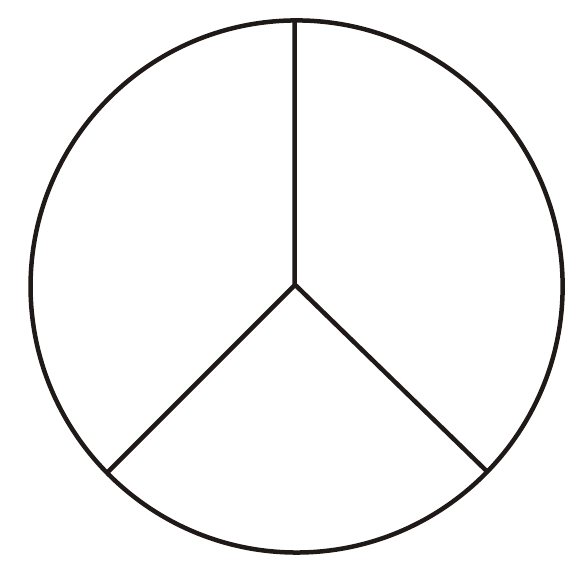
\caption[Feynman Graph of genus $3$]{Feynman graph $\Gamma_2$ of genus $3$.}\label{fig-stern}
\end{figure}
\[
I_{\Gamma_{2}}(q)=1152q^{6}+20736q^{8}+165888q^{10}+843264q^{12}+O(q^{14})
\]
leading to the representation%
\[
I_{\Gamma_{2}}(q)=\frac{24}{1492992}\left(  3 E_{4}^{3}-9 E_{2}%
^{2}E_{4}^{2}+9 E_{2}^{4}E_{4}-3 E_{2}^{6}\right)  \text{.}%
\]
Since $\Gamma_{1}$ and $\Gamma_{2}$ are the only Feynman graphs of genus $g=3$
without bridges, this yields the total generating function $F_{3}(q)$ for
the Hurwitz numbers $N_{d,3}$ as the $|\Aut(\Gamma_{i})|^{-1}$-weighted sum of
$I_{\Gamma_{1}}(q)$ and $I_{\Gamma_{2}}(q)$:
\begin{align*}
 F_{3}(q)  & =\frac{1}{16}I_{\Gamma_{1}}(q)+\frac{1}{24}I_{\Gamma_{2}%
}(q)\\
& =\frac{1}{1492992}\left(  4 E_{6}^{2}+7 E_{4}^{2}-12
E_{2}E_{4}E_{6}-12 E_{2}^{2}E_{4}^{2}+4 E_{2}^{3}E_{6}+15
E_{2}^{4}E_{4}-6 E_{2}^{6}\right)  \text{.}%
\end{align*}
This agrees with the formula for $F_{3}(q)$ given in \cite[Section 3]{Dij95} (up to a factor of $4!$ explained by the choice of convention mentioned in Remark \ref{rem-branchpointsmarked}).
\end{example}

\section{Appendix: Correspondence theorem}\label{ap-corres}
We now prove Theorem \ref{thm-corres}.
We do this by cutting covers of $E$ at the preimages of the base point $p_0$,
thus producing a (possibly reducible) cover of the tropical line
$\mathbb{TP}^1:=\RR\cup\{-\infty\}\cup\{\infty\}$ as in
\cite{CJM10}. We use the Correspondence Theorem of \cite{CJM10} relating the
numbers of tropical covers to certain tuples of elements of the symmetric group
that correspond to algebraic covers of $\mathbb{P}^1$.
We study ``gluing factors'' that relate the tropical multiplicity of a cover of
$E$ to the tropical multiplicity of the cut cover. These factors equal the
number of ways to produce a tuple of elements of the symmetric group
corresponding to a cover of an elliptic curve $\mathcal{E}$ from a tuple
corresponding to a cover of $\mathbb{P}^1$.


By pairing a cover of $\mathcal{E}$ with a monodromy representation, the Hurwitz number $N_{d,g}$ of Definition \ref{def-alghurwitz} equals the following count of tuples of permutations:

\begin{remark}[cf. \cite{RY10}]\label{eqn-ellipticCount}
The Hurwitz numbers $N_{d,g}$ of Definition \ref{def-alghurwitz} are given by $\frac{1}{d!}$
times the number of tuples $(\tau_1,\ldots,\tau_{2g-2},
\alpha,\sigma)$ of permutations in $\Perm_d$ such that
\begin{enumerate}
\item the $\tau_i$ are transpositions for all $i=1,\ldots,2g-2$,
\item  the equation $\tau_{2g-2}\circ\ldots\circ\tau_1\circ
\sigma=\alpha\circ\sigma\circ\alpha^{-1}$ holds,
\item the subgroup $\langle\tau_1,\ldots,\tau_{2g-2},\sigma,\alpha\rangle$ acts
transitively on $\{1,\ldots,d\}$.
\end{enumerate}
\end{remark}
Condition (2) is explained by Figure \ref{fig-monodrom} sketching the generators
of the fundamental group $\pi_1(\mathcal{E}\setminus{p_1,\ldots,p_{2g-2}})$:
The rectangle represents the elliptic curve as a torus, the left side is
identified with the right side and the top with the bottom. Clearly $\alpha$
and $\sigma$ belong to the two generators of $\pi_1(\mathcal{E})$. The loops
around the branch points $p_1,\ldots,p_{2g-2}$ induce the permutations
$\tau_1,\ldots,\tau_{2g-2}$ of sheets on the source curve. Obviously going
around all branch points simultaneously is homotopic to the composition of the
paths around each of the points in order $p_1,\ldots,p_{2g-2}$, on the one
hand, and on the other hand it is homotopic to the path that goes clockwise
around the rectangle. This identity has to ``lift'' to the source curve,
implying that
$\tau_{2g-2}\circ\ldots\circ\tau_1=\alpha\circ\sigma\circ\alpha^{-1}\circ
\sigma^{-1}$.

\begin{figure}
\begin{center}
 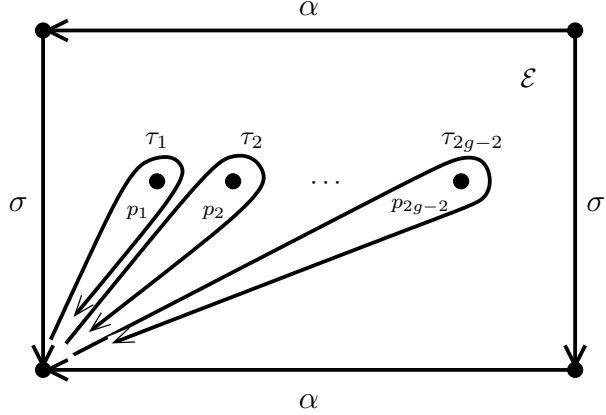
\end{center}

 \caption{The generators of the fundamental group $\pi_1(\mathcal{E}\setminus{p_1,\ldots,p_{2g-2}})$.}
\label{fig-monodrom}
\end{figure}

Let $\Delta$ be a partition of $d$. A permutation $\sigma$ has cycle type $\Delta$ if after decomposition into a disjoint union of cycles $\Delta$ equals the partition of the lengths of the cycles.
The conjugacy classes of $\Perm_d$ are the sets
of permutations with same cycle type $\Delta$. Thus
$\sigma$ and $\alpha\circ\sigma\circ\alpha^{-1}$ have the same cycle type for any $\sigma$ and $\alpha$. 

Given a tuple as in Remark \ref{eqn-ellipticCount} we now construct an associated tropical cover.
As before, we fix the base point $p_0$ and the  $2g-2$ branch points
$p_1,\ldots,p_{2g-2}$ (ordered clockwise) in $E$.

\begin{construction}\label{rem-perm2cover}
 Given a tuple as in Remark \ref{eqn-ellipticCount} we construct a tropical
cover of $E$ with branch points $p_1,\ldots,p_{2g-2}$ as follows:
\begin{enumerate}
 \item For each cycle $c$ of $\sigma$ of length $m$ draw an edge of weight
$m$ over $p_0$ and label it with the corresponding cycle.
 \item For $i=1,\ldots,2g-2$, successively cut or join edges over $p_i$
according to the effect of $\tau_i$ on
$\tau_{i-1}\circ\ldots\circ\tau_1\circ\sigma$. Label the
new edges as before.
 \item Glue the outcoming edges attached to points over $p_{2g-2}$ with the
edges over
$p_0$ according to the action of $\alpha$ on the cycles of $\sigma$. More
precisely: Glue the edge with the label $\alpha\circ c\circ\alpha^{-1}$ over
$p_{2g-2}$ to the edge with label $c$ over $p_0$.
 \item Forget all the labels on the edges.
\end{enumerate}
\end{construction}
Note that for a cycle $c=(n_1 \ldots n_l)$ of length $l\geq 2$ we have $\alpha \circ c \circ\alpha^{-1}=(\alpha(n_1) \ldots \alpha(n_l) )$. We use the same convention for cycles of length $1$.

\begin{example}
 Let $g=2$ and $d=4$ and consider the tuple of permutations
$$(\tau_1,\tau_2,\tau_3,\tau_4,\alpha,\sigma)=((1\,3),(2\,4),(1\,2),(1\,3),(2\,
3\,4),(2\,3))$$ in $\Perm_4$. We see that $\sigma=(2\,3)=(1)(2\,3)(4)$ has cycle
type $(2,1,1)$. Moreover
$\alpha\circ\sigma\circ\alpha^{-1}=(1)(2)(3\,
4)=\tau_4\circ\ldots\circ\tau_1\circ\sigma$ is fulfilled, so the tuple
contributes to the count of $N_{4,2}$. Figure \ref{fig-perm2cover1}
sketches the construction of Remark \ref{rem-perm2cover} up to the gluing step
(the object can be considered as a cover of the tropical line $\mathbb{TP}^1$
as described below).
\begin{figure}
 \begin{center}
 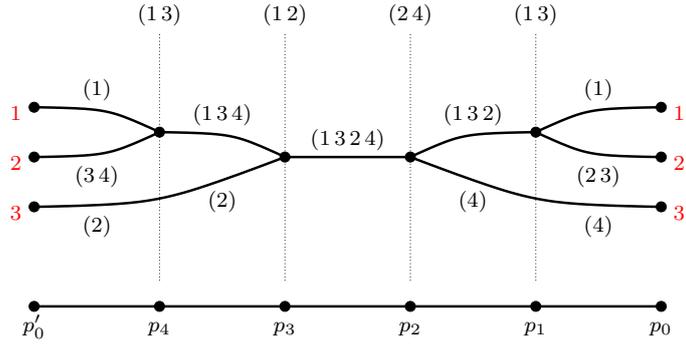
 \end{center}
\caption{The cover of $\mathbb{TP}^1$ associated to a tuple of
permutations.}\label{fig-perm2cover1}
\end{figure}

In the gluing step the vertices $p_0$ and $p_0'$ are going to be identified.
Since we have
\begin{eqnarray*}
 \alpha\circ(1)\circ\alpha^{-1}&=&(1)\\
_\wedge\  \alpha\circ(2\,3)\circ\alpha^{-1}&=&(3\,4)\\
_\wedge\ \ \ \alpha\circ(4)\circ\alpha^{-1}&=&(2),
\end{eqnarray*}
the ends of the source curve are glued according to the red numbers in Figure
\ref{fig-perm2cover1}. The result is the cover of $E$ depicted on the left in
Figure \ref{fig-perm2cover2}. Note that choosing $\alpha=(2\,4)$ yields the
same gluing, while $\alpha'=(1\,2\,4)$ also fulfills
$\alpha'\circ\sigma\circ\alpha'^{-1}=\tau_4\circ\ldots\circ\tau_1\circ\sigma$,
but since
\begin{eqnarray*}
 \alpha'\circ(1)\circ\alpha'^{-1}&=&(2)\\
_\wedge\  \alpha'\circ(2\,3)\circ\alpha'^{-1}&=&(3\,4)\\
_\wedge\ \ \ \alpha'\circ(4)\circ\alpha'^{-1}&=&(1),
\end{eqnarray*} it provides
 a different gluing, sketched on the right side of Figure
\ref{fig-perm2cover2}. In particular the combinatorial types of the source
curves are different.
\begin{figure}
 \begin{center}
 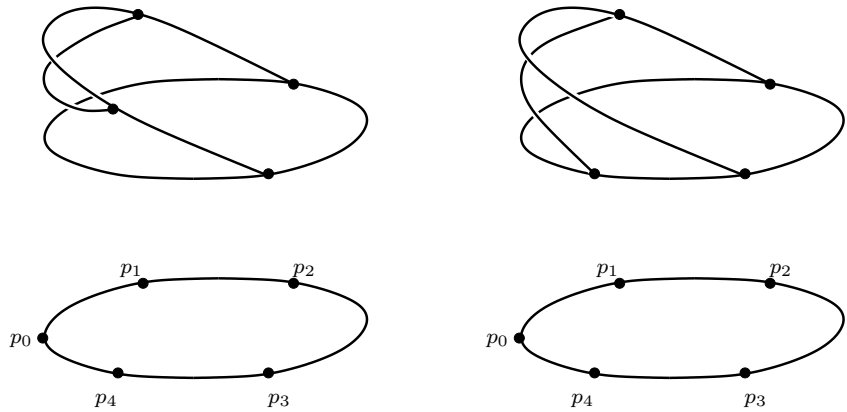
 \end{center}
\caption[Two covers of $E$ associated to tuples of
permutations]{Two different gluings of the same cover of the line.}\label{fig-perm2cover2}
\end{figure}

\end{example}

Now we describe how to cut tropical covers of $E$ in general, thus producing covers of a line. As usual, we neglect edge lengths --- to be precise, they have to adapted accordingly.
\begin{construction}\label{const-cutting}
To every cover $\pi:C\rightarrow E$ of degree $d$ we associate a (possibly disconnected) tropical cover 
$\tilde{\pi}:\tilde{C}\rightarrow\mathbb{TP}^1$ of the line
$\mathbb{TP}^1$ of the same degree,
by cutting $E$ at $p_0$ and the source curve $C$ at every preimage of $p_0$.
\end{construction}

For the definition of (connected) tropical covers of $\mathbb{TP}^1$, see \cite{CJM10}. This definition can easily be generalized by allowing the source curve to be disconnected, and adapting the multiplicity accordingly as follows.
Notice that the multiplicity of a single edge of weight $m$ covering $\mathbb{TP}^1$  is $\frac{1}{m}$ (this case is not taken care of in \cite{CJM10} where it is implicitly assumed that every source curve contains at least one vertex).
The multiplicity of a cover $\tilde{\pi}:\tilde{C}\rightarrow\mathbb{TP}^1$ is  \begin{equation}\mult(\tilde{\pi}):= \prod_K \frac{1}{w_K} \cdot |\Aut(\tilde{\pi})|^{-1} \cdot\prod_e w_e,\label{eq-multcoverp1}\end{equation}
 where the first product goes over all connected components $K$ of $\tilde{C}$ that just consist of one single edge mapping to $\mathbb{TP}^1$ with weight $w_K$ and the second product goes over all bounded edges $e$ of $\tilde{C}$, with $w_e$ denoting their weight.
Note that in \cite{CJM10}, the factor $|\Aut(\tilde{\pi})|^{-1}$ is simplified to
$\frac{1}{2}^{l_1+l_2}$, where $l_1$ denotes the number of balanced forks (i.e.,
adjacent ends of the same weight) and $l_2$ denotes the number of wieners (i.e.
pairs of bounded edges of the same weight sharing both end vertices). Since we
allow disconnected covers, we will have other contributions to the automorphism
group: connected components consisting of single edges of the same weight as
above can be permuted. So we get a contribution of $\frac{1}{r!}$ to
$|\Aut(\tilde{\pi})|$ if for a certain weight $m$ there are exactly $r$ copys of
connected components consisting of a single edge of weight $m$.

For a cover $\pi:C\rightarrow E$, we denote by
$\Delta$ the partition
of $d$ given by the weights of the edges over $p_0$. For the cut cover $\tilde{\pi}$ (see Construction \ref{const-cutting}) these
are exactly the ramification profiles over $-\infty$ and $\infty$.
\begin{example}
The two covers of $E$ depicted in Figure \ref{fig-perm2cover2} cut at $p_0$ both
give a cover of the line as sketched in Figure \ref{fig-perm2cover1} (where we dropped the labels on the edges).
The multiplicity of the cover of the line equals the product of the weight of the bounded edges, i.e., $3^2\cdot4$, since there are no automorphisms.
\end{example}
\begin{example}
Figure \ref{fig-coverOfLine} shows a disconnected cover $\tilde\pi$ of the line of degree
$17$ with $6$ (simple) ramifications and profiles $(4,4,4,1,1,1,1,1)$ over the
ends. Its multiplicity equals
$$\mult(\tilde\pi)=
\underbrace{
\left(\frac{1}{2}\right)^2
\cdot
\frac { 1 } { 2 }
\cdot
\frac{1}{2!}
\cdot
\frac{1}{3!}
}_{=|\Aut(\tilde\pi)|}
\cdot
\underbrace{\frac{1}{4\cdot4\cdot1\cdot1\cdot1}}_{=\prod_K \frac{1}{w_K}}
\cdot
\underbrace{4\cdot2\cdot2\cdot2\cdot2}_{=\prod_e w_e}
=\frac{1}{24},
$$
where the first factor contributing to the automorphisms comes from the two
balanced forks, the second from the wiener and the other two from two
single-edge components of weight $4$ and three single-edge components of weight $1$ respectively.
\begin{figure}
 \begin{center}
  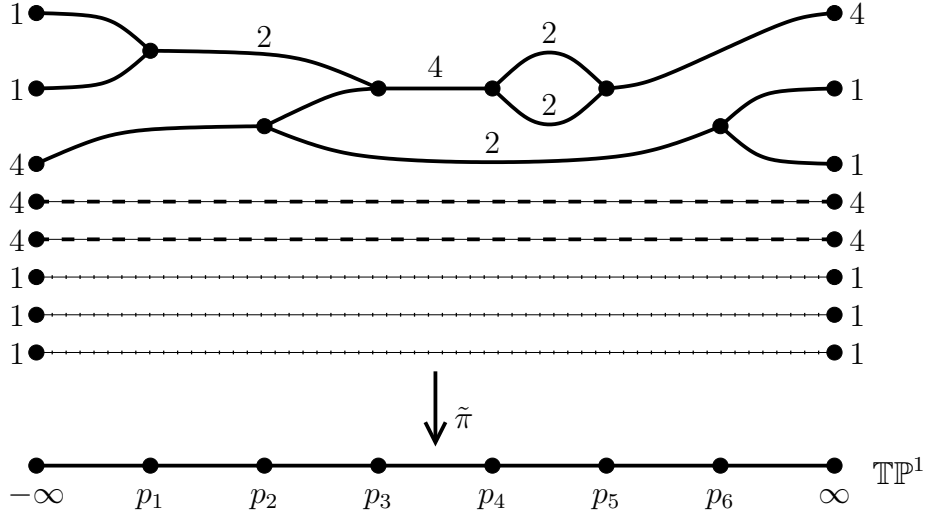
 \end{center}
\caption{A (disconnected) tropical cover of the
line.}\label{fig-coverOfLine}
\end{figure}
\end{example}

The Correspondence Theorem in \cite{CJM10} matches tropical covers of $\mathbb{TP}^1$ as above with algebraic covers of $\PP^1$ having two
ramifications of profile $\Delta$ over $0$ and $\infty$ respectively and only
simple
ramifications else. 

Similar to Remark \ref{eqn-ellipticCount}, the associated
Hurwitz numbers can be written in terms of tuples of elements of the symmetric group.
\begin{remark}\label{rem-doubleHurwitz}
 The double Hurwitz number $H_{d,g}(\PP^1,\Delta,\Delta)$ counting the number of
(isomorphism classes of) covers $\phi:\mathcal{C}\rightarrow \PP^1$ of degree
$d$ (each weighted with $\frac{1}{|\Aut(\phi)|}$), where $\mathcal{C}$ is a
possibly disconnected curve such that the sum of the genera of its connected
components equals $g$, having
ramification profile $\Delta$ over $0$ and $\infty$ and only simple
ramifications else, equals
$\frac{1}{d!}$ times the number of tuples
$(\tau_1,\ldots,\tau_{2g-2},
\sigma,\sigma')$ in $\Perm_d$ such that
\begin{itemize}
\item $\sigma$ and $\sigma'$ are permutations of cycle
type $\Delta$,
\item the $\tau_i$ are transpositions for all $i=1,\ldots,2g-2$,
\item the equation $\sigma'\circ\tau_{2g-2}\circ\ldots\circ\tau_1\circ
  \sigma=\id_{\Perm_d}$ holds.
\end{itemize}
Note that as in Definition \ref{def-alghurwitz}, it follows from the Riemann-Hurwitz formula that the number of simple ramifications is $2g-2$.
The condition $\sigma'\circ\tau_{2g-2}\circ\ldots\circ\tau_1\circ
  \sigma=\id_{\Perm_d}$ reflects the fact that the fundamental group $\pi_1(\PP^1)$ is trivial.
We do not include a condition about transitivity here, since we allow also disconnected covers.
\end{remark}

As in Construction \ref{rem-perm2cover}, we can associate a tropical cover of the line to a tuple as in Remark \ref{rem-doubleHurwitz}. 
The procedure is the same, we just drop the gluing step (3).
The statement of the Correspondence Theorem 5.28 in \cite{CJM10} is that for a fixed tropical cover $\tilde{\pi}:\tilde{C}\rightarrow \mathbb{TP}^1$, the tropical multiplicity equals $\frac{1}{d!}$ times the number of tuples that yield $\tilde{\pi}$ under the above procedure.

We now relate the tuples in Remarks \ref{eqn-ellipticCount} and \ref{rem-doubleHurwitz}. resp.\ the multiplicities of a tropical cover $\pi:C\rightarrow E$ and the cut cover $\tilde{\pi}:\tilde{C}\rightarrow \mathbb{TP}^1$ of Construction \ref{const-cutting}.
\begin{definition}\label{def-npipi}
 Given a cover $\pi:C\rightarrow E$ and the cut cover $\tilde{\pi}:\tilde{C}\rightarrow \mathbb{TP}^1$ of Construction \ref{const-cutting},
we choose a tuple $(\tau_1,\ldots,\tau_{2g-2},
\sigma,\sigma')$ that yields $\tilde\pi$ when applying Construction \ref{rem-perm2cover} (minus the gluing in step (3)).
We define $n_{\tilde\pi,\pi}$ to be the number of $\alpha\in \Perm_d$ satisfying
$\alpha\circ\sigma\circ\alpha^{-1}=\sigma'$ and, when labeling $\tilde{\pi}$
with cycles according to our choice of tuple
$(\tau_1,\ldots,\tau_{2g-2},\sigma,\sigma')$ and performing step (3) and (4) of
Construction \ref{rem-perm2cover} (gluing and forgetting the cycle labels), we
obtain $\pi$.
\end{definition}
Note that $n_{\tilde\pi,\pi}$ is well-defined (i.e., does not depend on the choice of the tuple $(\tau_1,\ldots,\tau_{2g-2},\sigma,\sigma')$). This is true since any other representative $(\bar\tau_1,\ldots,\bar\tau_{2g-2},
\bar\sigma,\bar\sigma')$ is a conjugate of $(\tau_1,\ldots,\tau_{2g-2},
\sigma,\sigma')$ and therefore the desired $\bar\alpha$ are in one-to-one
correspondence to the desired $\alpha$.


\begin{figure}
  \begin{center}
   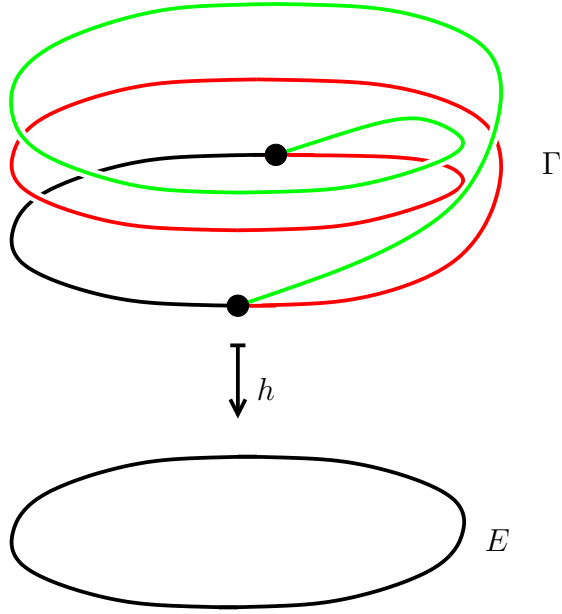
  \end{center}
\caption{A tropical elliptic cover with a long wiener. The two wiener-edges (i.e., the red and the green edge) have the same weight.}\label{fig-longWiener}
 \end{figure}

\begin{proposition}\label{prop-gluingFactor}
 For a cover $\pi:C\rightarrow E$ with partition $\Delta=(m_1,\ldots,m_r)$
over the base point, the number $n_{\tilde\pi,\pi}$ of Definition \ref{def-npipi} is given by
$$n_{\tilde\pi,\pi}=m_1\cdot\ldots\cdot m_r\cdot\frac{|\Aut(\tilde\pi)|}{|\Aut(\pi)|}.$$
\end{proposition}
\begin{proof}
As in the definition of $n_{\tilde\pi,\pi}$ (see Definition \ref{def-npipi}), fix a tuple of permutations
$(\tau_1,\ldots,\tau_{2g-2},\sigma,\sigma')$ that yields $\tilde\pi$ when applying Construction \ref{rem-perm2cover} minus the gluing step (3).

The set of
$\alpha$ such that $\alpha\circ\sigma\circ\alpha^{-1}=\sigma'$ is a coset of the
stabilizer of $\sigma$ with respect to the operation of $\Perm_d$ on itself via
conjugation:
$(\alpha,\sigma)\mapsto\alpha\circ\sigma\circ\alpha^{-1}$.
Assume that $\Delta$ consists of $k_i$ weights $w_i$ for $i=1,\ldots,s$,
then this stabilizer is isomorphic to the semidirect product $\prod_{i=1}^s
C_{w_i}^{k_i}\rtimes\prod_{i=1}^s\Perm_{k_i}$ of cyclic groups $C_{w_i}$ of
length $w_i$ and symmetric groups $\Perm_{k_i}$. 
This can be seen as follows: for each weight $w_i$ (i.e., length of a cycle of $\sigma$) we can choose an element of $\Perm_{k_i}$ permuting the cycles of length $w_i$ in $\sigma$. Assume the cycle $c_1$ of $\sigma$ is mapped to the cycle $c_2$ by this permutation. Then we consider permutations $\alpha'$ in the group of bijections of the entries of $c_2$ to the entries of $c_1$  that satisfy $\alpha' \circ c_1\circ\alpha'^{-1}=c_2$, there are $w_i$ such $\alpha'$ (and they form a cyclic group). Since the cycles of $\sigma$ are disjoint, the choices for $\alpha'$ for each pair of cycles $(c_1,c_2)$ where $c_1$ is mapped to $c_2$ under the permutations in $\Perm_{k_i}$ that we choose for each $i$ can be combined to a unique $\alpha$ in the stabilizer of $\sigma$.

We label the edges of $\tilde{C}$ with cycles as given by the choice of our tuple. 
Transferred to our situation, the argument above shows that when searching for $\alpha$ that satisfy both requirements of Definition \ref{def-npipi}, we always get the contributions from the $C_{w_i}$, (leading to a factor of $\prod_{i=1}^s w_i^{k_i}= m_1\cdot \ldots \cdot m_r$). To prove the lemma, it remains to see that $\frac{|\Aut(\tilde\pi)|}{|\Aut(\pi)|}$ equals the number of ways to choose permutations of the cycles of the same length (resp.\ permutations of the ends of $\tilde{C}$) that correspond to a gluing of $\tilde{\pi}$ equal to $\pi$ when applying Construction \ref{rem-perm2cover}, step (3).


 
So let us now analyze the automorphism groups and compare the quotient of their sizes to the possibilities to glue the cover $\tilde\pi$ (with labeled ends) to $\pi$.

The automorphism group of $\tilde \pi$ is, as mentioned above, a direct product
of symmetric groups each corresponding to a wiener, a balanced fork or the set
of connected components consisting of a single edge of fixed weight. 
The automorphism group of $\pi$ is a direct product of symmetric groups of size two corresponding to wieners.
Notice that we can have long wieners as in Figure \ref{fig-longWiener}, where the two edges of the same weight are curled equally.
Clearly automorphisms that come from wieners that are not cut cancel in the quotient and we can thus disregard them.
Since therefore all contributions to the automorphism groups we have to consider come from ends of $\tilde{C}$, and the possibilities to glue the cover $\tilde{\pi}$ to $\pi$ also only depend on the ends of $\tilde{C}$, we can analyze the situation locally on the level of the involved ends.

We say that an end of $\tilde{C}$ is distinguishable if it is not part of a balanced fork and not an
end of a component consisting of a single edge. Distinguishable ends do not
contribute to the automorphisms of $\tilde\pi$.

We have to consider several cases. We first consider cases not involving connected components consisting of a single edge.
\begin{enumerate}
 \item If we glue two distinguishable ends of $\tilde C$ to get back $C$, there are no choices for different gluings. Since distinguishable ends do not contribute to the automorphisms, the equality of contributions from these ends holds.
\item Assume that an edge of $C$ is cut in such a way that one of the ends
is part of a balanced fork and the other is distinguishable. Then obviously
there are $2$ ways to glue, see Figure \ref{fig-glueforktodist}.
\begin{figure}
\begin{center}
 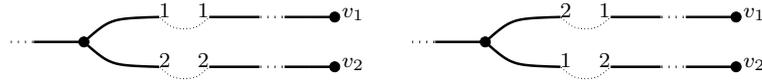
\end{center}
\caption{Two ways to glue a fork to two distinguishable ends.}
\label{fig-glueforktodist}
\end{figure}
The balanced fork contributes with a factor $2$ to
$|\Aut(\tilde\pi)|$.  After gluing, the fork is not part of a wiener, so the contribution to
$|\Aut(\pi)|$ is $1$. Again, we see that the contributions coming from these ends to the quotient of the sizes of the automorphism groups on the one hand and to the possibilities of gluing on the other hand coincide.
\item If two balanced forks are glued, we obtain a wiener.
The contribution
to $|\Aut(\tilde\pi)|$ and $|\Aut(\pi)|$ is $4$ and $2$ respectively. The ways
to glue
the forks to a wiener is $2$, as illustrated in Figure \ref{fig-gluefork}.
\end{enumerate}

\begin{figure}
 \begin{center}
 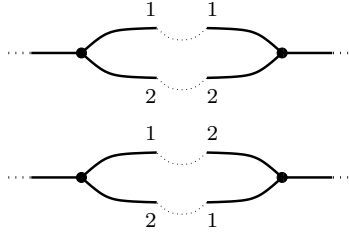
\end{center}
\caption{Gluing two balanced forks to a wiener.}\label{fig-gluefork}
\end{figure}

Now we have to consider cases involving ends of connected components consisting of a single edge, say of weight $m$.
Assume there are $l$ components consisting of a single edge of weight $m$.
These ends contribute a factor of $l!$ to
$|\Aut(\tilde\pi)|$.
\begin{enumerate}\setcounter{enumi}{3}
 \item Assume that $l_0$ of the components are not part of a long
wiener after gluing. They do not
contribute to the automorphisms of $\pi$. Note that the components nevertheless
might be attached to balanced forks. In this case the fork is either
part of a pseudo-wiener in $\pi$ (i.e., two edges sharing the same end vertices and having the same weight, but curled differently, see Figure \ref{pseudoWiener})
or the two edges of the fork
have different end vertices.
 \begin{figure}
 \begin{center}
  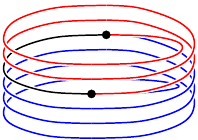
 \end{center}
 \caption{A cover with a pseudo-wiener.}\label{pseudoWiener}
 \end{figure}

Let us now determine the number of ways to glue these ends of $\tilde C$ to get back $C$.
We can choose $l_0$ of the $l$ single-edge-components, and
distribute them to $l_0$ distinguishable places. Also, we
get a factor of $2$ for each balanced
fork involved. The result is $\binom{l}{l_0}\cdot l_0!\cdot 2^{f}$, where
$f$ is the number of balanced forks involved.

The remaining $l-l_0$ components must be part of long wieners in $\pi$.
Let $n$ be the number of long wieners in $\pi$, giving a contribution of
$2^n$ to $|\Aut(\pi)|$, then $2n$ balanced forks from $\tilde\pi$
are involved in the gluing process contributing with a factor of
$2^{2n}$ to $|\Aut(\tilde\pi)|$. The number of ways to glue is the
number of ways to distribute the $l-l_0$ components to $l-l_0$ gluing places and a
factor of $2$ for every wiener we get, just as in Figure \ref{fig-gluefork}.
Altogether the contribution to the number of gluings providing the 
desired cover equals $$\binom{l}{l_0}\cdot l_0!\cdot
2^{f}\cdot(l-l_0)!\cdot2^n=l!\cdot2^f\cdot2^n.$$ The contribution to the quotient of
the sizes of the automorphism groups equals $$\frac{l!\cdot 2^f\cdot 2^{2n}}{2^n}.$$ Obviously, the two expressions coincide and we are done.
\end{enumerate}
\end{proof}

We are now ready to prove the Correspondence Theorem \ref{thm-corres}, that is
the equality of tropical and algebraic Hurwitz numbers of simply ramified covers of elliptic
curves.
\begin{proof}[Proof of Theorem \ref{thm-corres}]
 By Remark \ref{eqn-ellipticCount} 
$$
 N_{d,g}
=
\frac{1}{d!}\cdot\#\left\{(\tau_1,\ldots,\tau_{2g-2},
\alpha,\sigma)\right\},$$
where $\alpha,\sigma, \tau_i\in\Perm_d$, the $\tau_i$ are transpositions, the equality $\tau_{2g-2}\circ\ldots\circ\tau_1\circ
\sigma=\alpha\circ\sigma\circ\alpha^{-1}$ holds and $\langle\tau_1,\ldots,\tau_{2g-2},\sigma\rangle$ acts
transitively on the set $\{1,\ldots,d\}$.
We can group the tuples in the set according to the tropical cover $\pi:C\rightarrow E$ they provide under Construction \ref{rem-perm2cover} and write the sum above as
$$
\frac{1}{d!}\cdot \sum_{\pi}\#\left\{(\tau_1,\ldots,\tau_{2g-2},
\alpha,\sigma)\textnormal{ yielding the cover }\pi\right\}.$$
For a fixed
cover $\pi$, instead of counting tuples yielding $\pi$, we can count tuples $(\tau_1,\ldots,\tau_{2g-2},\sigma,\sigma')$ yielding the cut cover
$\tilde\pi$ from Construction \ref{const-cutting} and then multiply with the number of appropriate $\alpha$, i.e., with
$n_{\tilde\pi,\pi}$ (see Definition \ref{def-npipi}):
$$
\frac{1}{d!}\cdot\sum_\pi
\#\left\{(\tau_1,\ldots,\tau_{2g-2},
\sigma,\sigma')\textnormal{ that provide the cover }\tilde{\pi}\right\}\cdot
n_{\tilde{\pi},\pi}.$$
By \cite{CJM10} (see also Remark \ref{rem-doubleHurwitz}) the count of the
tuples yielding a cover $\tilde\pi$ divided by $d!$ coincides with its tropical
multiplicity $\mult(\tilde\pi)=\frac{1}{|\Aut(\tilde{\pi})|}\cdot\prod_K
\frac{1}{w_K}\cdot\prod_{\tilde
e} w_{\tilde e}$ where the first product goes over all components $K$ consisting of a single edge of weight $w_K$ and the second product goes over all bounded edges $\tilde{e}$ of $\tilde{C}$ and $w_{\tilde{e}}$ denotes their weight (see (\ref{eq-multcoverp1})). Using
Proposition \ref{prop-gluingFactor}, the number $n_{\tilde{\pi},\pi}$ can be
substituted by
$\prod_{e'}w_{e'}^{c_{e'}}\cdot\frac{|\Aut(\tilde\pi)|}{|\Aut(\pi)|}$ where the
product goes over all edges $e'$ of $C$ that contain a preimage of the base
point $p_0$ of $E$ and $c_{e'}$ denotes the number of preimages in $e'$,
$c_{e'}= \#  (\pi^{-1}(p_0)\cap e') $.
We obtain
$$
N_{d,g}=\sum_\pi
\frac{1}{|\Aut(\tilde{\pi})|}\cdot\prod_{\tilde{e}}w_{\tilde{e}}
\cdot\prod_K \frac{1}{w_K}\cdot\prod_ { e' }
w_{e'}^{c_{e'}}\cdot\frac{|\Aut(\tilde\pi)|}{|\Aut(\pi)|}.
$$
An edge $e'$ of $C$ of weight $w_{e'}$ having $c_{e'}$
preimages over the base point provides exactly $c_{e'}-1$ single-edge-components of weight
$w_{e'}$  in the cut cover $\tilde\pi$. Vice versa, each such
component comes from an edge with multiple preimages over the
base point. Therefore the expression $\prod_K \frac{1}{w_K}\cdot\prod_ { e' }
w_{e'}^{c_{e'}}$ simplifies to $\prod_{e'} w_{e'}$. We obtain
$$
N_{d,g}=\sum_\pi
\frac{1}{|\Aut(\pi)|}\prod_e w_e = N_{d,g}^{\trop}$$
 and the theorem is proved.
\end{proof}

\bibliographystyle {plain}
\bibliography {bibliographie}

\end{document}